\documentclass[reqno,10pt,centertags,draft]{amsart}
\usepackage{amsmath,amsthm,amscd,amssymb,latexsym,esint,upref,stmaryrd,
enumerate,color,verbatim,yfonts}
\usepackage{hyperref} 
\newcommand*{\mailto}[1]{\href{mailto:#1}{\nolinkurl{#1}}}
\newcommand{\arxiv}[1]{\href{http://arxiv.org/abs/#1}{arXiv:#1}}

\newcommand{\bbC}{{\mathbb{C}}}

\newcommand{\bbN}{{\mathbb{N}}}

\newcommand{\bbR}{{\mathbb{R}}}

\newcommand{\bbZ}{{\mathbb{Z}}}

\newcommand{\beq}{\begin{equation}}
\newcommand{\enq}{\end{equation}}

\newcommand{\g}{\gamma}

\DeclareMathOperator{\supp}{supp}
 
\DeclareMathOperator{\rank}{rank}

\DeclareMathOperator{\dom}{dom}

\renewcommand{\Re}{\text{\rm Re}}
\renewcommand{\Im}{\text{\rm Im}}
\renewcommand{\ln}{\text{\rm ln}}

\newcommand{\no}{\notag}
\newcommand{\lb}{\label}
\newcommand{\f}{\frac}

\newcommand{\ol}{\overline}
\newcommand{\bs}{\backslash}

\newcommand{\wti}{\widetilde}
\newcommand{\Oh}{O}
\newcommand{\oh}{o}
\newcommand{\hatt}{\widehat} 
\newcommand{\dott}{\,\cdot\,}

\newcommand{\bi}{\bibitem}

\let\geq\geqslant
\let\leq\leqslant

\newcommand{\la}{\lambda}
\newcommand{\lam}{\lambda}

\newcommand{\al}{\alpha}
\newcommand{\be}{\beta}
\newcommand{\ga}{\gamma}

\newcommand{\vl}{\big\vert}
\newcommand{\Lr}{{L^2((a,b);rdx)}} 
\newcommand{\AC}{{AC([a,b])}}
\newcommand{\ACl}{{AC_{loc}((a,b))}}
\newcommand{\Ll}{{L^1_{loc}((a,b);dx)}}

\makeatletter
\def\theequation{\@arabic\c@equation}

\allowdisplaybreaks 
\numberwithin{equation}{section}

\newtheorem{theorem}{Theorem}[section]

\newtheorem{lemma}[theorem]{Lemma}

\newtheorem{definition}[theorem]{Definition}
\newtheorem{hypothesis}[theorem]{Hypothesis}

\theoremstyle{remark}
\newtheorem{remark}[theorem]{Remark}

\begin{document}

\title[Singular Sturm--Liouville Operators]{On Self-Adjoint Boundary Conditions for Singular Sturm--Liouville Operators Bounded From Below} 

\author[F.\ Gesztesy]{Fritz Gesztesy}
\address{Department of Mathematics, 
Baylor University, One Bear Place \#97328,
Waco, TX 76798-7328, USA}
\email{\mailto{Fritz\_Gesztesy@baylor.edu}}
\urladdr{\url{http://www.baylor.edu/math/index.php?id=935340}}

\author[L.\ L.\ Littlejohn]{Lance L. Littlejohn}
\address{Department of Mathematics, 
Baylor University, One Bear Place \#97328,
Waco, TX 76798-7328, USA}
\email{\mailto{Lance\_Littlejohn@baylor.edu}}
\urladdr{\url{http://www.baylor.edu/math/index.php?id=53980}}

\author[R.\ Nichols]{Roger Nichols}
\address{Department of Mathematics, The University of Tennessee at Chattanooga, 
415 EMCS Building, Dept.~6956, 615 McCallie Ave, Chattanooga, TN 37403, USA}
\email{\mailto{Roger-Nichols@utc.edu}}
\urladdr{\url{http://www.utc.edu/faculty/roger-nichols/index.php}}


\date{\today}

\begin{abstract} 
We extend the classical boundary values 
\begin{align}
\begin{split} 
g(a) &= - W(u_a(\lambda_0, \dott), g)(a) =  \lim_{x \downarrow a} \f{g(x)}{\hatt u_a(\lambda_0,x)},   \lb{23.0.1} \\
g^{[1]}(a) &= (p g')(a) = W(\hatt u_a(\lambda_0, \dott), g)(a) 
= \lim_{x \downarrow a} \f{g(x) - g(a) \hatt u_a(\lambda_0,x)}{u_a(\lambda_0,x)},    
\end{split} 
\end{align}
for regular Sturm--Liouville operators associated with differential expressions of the type 
$\tau = r(x)^{-1}[-(d/dx)p(x)(d/dx) + q(x)]$ for a.e.~$x\in(a,b) \subset \bbR$, to the case where $\tau$ is singular 
on $(a,b) \subseteq \bbR$ and the associated minimal operator $T_{min}$ is bounded from below. Here 
$u_a(\lambda_0, \dott)$ and $\hatt u_a(\lambda_0, \dott)$ denote suitably normalized principal and nonprincipal solutions of 
$\tau u = \lambda_0 u$ for appropriate $\lambda_0 \in \bbR$, respectively.

Our approach to deriving the analog of \eqref{23.0.1} in the singular context employing principal and nonprincipal solutions of $\tau u = \lambda_0 u$ is closely related to a seminal 1992 paper by Niessen and Zettl \cite{NZ92}. We also recall the well-known fact that the analog of the boundary values in \eqref{23.0.1} characterizes all self-adjoint extensions of $T_{min}$ in the singular case in a manner familiar from the regular case. 

We briefly discuss the singular Weyl--Titchmarsh--Kodaira $m$-function and finally illustrate the theory in some detail with the examples of the Bessel, Legendre, and Kummer (resp., Laguerre) operators.
\end{abstract}

\maketitle


{\scriptsize{\tableofcontents}}

\section{Introduction} \lb{23.s1} 

The principal purpose of this paper is to introduce generalized boundary values for singular Sturm--Liouville 
operators on arbitrary intervals $(a,b) \subseteq \bbR$, bounded from below, generated by differential expressions 
of the type 
\begin{equation}
\tau = \f{1}{r(x)}\left[-\f{d}{dx}p(x)\f{d}{dx} + q(x)\right] \, \text{ for a.e.~$x\in(a,b) \subseteq \bbR$,} 
   \lb{23.1.1}
\end{equation} 
which naturally extend the ones in the case where $\tau$ is regular, and which are equivalent (in fact, equal) to the standard ones defined in terms of Wronskians involving a fundamental system of principal and nonprincipal solutions 
of $\tau u = \lambda_0 u$ for appropriate $\lambda_0 \in \bbR$. 

To describe this in some detail we focus on the left endpoint $a$ for simplicity and assume temporarily that $\tau$ is regular on the (then necessarily finite) interval 
$(a,b)$, and that $\hatt u_a$ is a normalized solution of $\tau u = \lambda_0 u$ satisfying
\begin{equation}
\hatt u_a(\lambda_0,a) =1, \quad [p(x) \hatt u_a'(\lambda_0,x)]|_{x=a} = \hatt u_a^{[1]}(\lambda_0,a) = 0  
\end{equation}
($u^{[1]}$ denoting the first quasi-derivative of $u$).  Given $\hatt u_a$, we introduce a second, linearly independent solution of $\tau u = \lambda_0 u$, by the variation of parameters method, 
\begin{equation}
u_a(\lambda_0,x) = \hatt u_a(\lambda_0,x) \int_a^x dx' \, p(x')^{-1} [\hatt u_a(\lambda_0,x')]^{-2}, 
\end{equation}
such that
\begin{equation}
u_a(\lambda_0,a) =0, \quad [p(x) u_a'(\lambda_0,x)]|_{x=a} = u_a^{[1]}(\lambda_0,a) = 1, 
\end{equation}
and hence the Wronskian of $\hatt u_a$ and $u_a$ is normalized as well, 
\begin{equation}
W\big(\hatt u_a(\lambda_0, \dott), u_a(\lambda_0, \dott)\big) 
=  \hatt u_a(\lambda_0,a)  u_a^{[1]}(\lambda_0,a) - \hatt u_a^{[1]}(\lambda_0,a) u_a(\lambda_0,a) = 1.
\end{equation}
Given these facts, we now take a closer look at the classical boundary values in the case where $\tau$ is regular 
at the endpoint $a$,
\begin{equation}
g(a) \, \text{ and } \, g^{[1]}(a) \, \text{ for $g \in \dom(T_{max})$,}
\end{equation}
where $T_{max}$ abbreviates the maximal operator in $L^2((a,b); rdx)$ associated with $\tau$ (see Sections \ref{23.s2} 
and \ref{23.s3} for properties of elements $g \in \dom(T_{max})$).  Then one computes, 
\begin{align}
\begin{split} 
- W(u_a(\lambda_0, \dott), g)(a) 
&= \lim_{x \downarrow a} \big[u_a^{[1]}(\lambda_0,x) g(x) - u_a(\lambda_0,x) g^{[1]}(x)\big] 
= g(a)    \\
&= \lim_{x \downarrow a} \f{g(x)}{\hatt u_a(\lambda_0,x)}.
\end{split} 
\end{align} 
In addition, one obtains
\begin{align}
W(\hatt u_a(\lambda_0, \dott), g)(a) 
&= \lim_{x \downarrow a} \big[\hatt u_a(\lambda_0,x) g^{[1]}(x) - \hatt u_a^{[1]}(\lambda_0,x) g(x)\big]  \no \\
&= \lim_{x \downarrow a} \hatt u_a(\lambda_0,x) g^{[1]}(x) \no \\
&= g^{[1]}(a). 
\end{align}
On the other hand, by L'H{\^ o}pital's rule,
\begin{align}
& \lim_{x \downarrow a} \f{g(x) - g(a) \hatt u_a(\lambda_0,x)}{u_a(\lambda_0,x)} 
= \f{0}{0} =  \lim_{x \downarrow a} \f{g'(x) - g(a) \hatt u_a'(\lambda_0,x)}{u_a'(\lambda_0,x)} 
\cdot \f{p(x)}{p(x)}     \no \\
& \quad =  \lim_{x \downarrow a} \f{g^{[1]}(a) + o(1)}{u_a^{[1]}(\lambda_0,a) + \oh(1)}   \no \\ 
&\quad = g^{[1]}(a).   \lb{23.1.9} 
\end{align}
Summarizing, the two classical boundary values $g(a)$ and $g^{[1]}(a)$ in the case where $\tau$ is regular on $[a,b]$, satisfy,
\begin{align}
g(a) &= - W(u_a(\lambda_0, \dott), g)(a) =  \lim_{x \downarrow a} \f{g(x)}{\hatt u_a(\lambda_0,x)},   \lb{23.1.11} \\
g^{[1]}(a) &= W(\hatt u_a(\lambda_0, \dott), g)(a) 
= \lim_{x \downarrow a} \f{g(x) - g(a) \hatt u_a(\lambda_0,x)}{u_a(\lambda_0,x)}.    \lb{23.1.12}   
\end{align}

The main purpose of this paper is to extend \eqref{23.1.11}, \eqref{23.1.12} to the case where $\tau$ is singular on $(a,b)$, and the associated minimal operator $T_{min}$ is bounded from below, by identifying $u_a(\lambda_0, \dott)$ and 
$\hatt u_a(\lambda, \dott)$ as {\it principal} and {\it nonprincipal} solutions of $\tau u = \lambda_0 u$ (cf.\ Section \ref{23.s4} for their precise definition), respectively. In particular, we recall in Theorem \ref{23.t7.3.11} how the use of the resulting generalized boundary values characterized, for instance, by the right-hand sides of \eqref{23.1.11}, \eqref{23.1.12}, describes 
all self-adjoint extensions of  $T_{min}$. At this point it seems fair to stress that the use of principal and nonprincipal solutions in this context was pioneered by Rellich \cite{Re51} in 1951 (see also his paper \cite{Re43}, which foreshadowed the one in 1951) and later used to perfection in the 1992 paper by Niessen and Zettl \cite{NZ92}.  

The first equalities in either of \eqref{23.1.11} and \eqref{23.1.12} are well-known in the singular context, see, for instance, 
\cite{NZ92}. Moreover, we emphasize that the second equality in \eqref{23.1.11} was derived in the seminal paper by Niessen and Zettl \cite{NZ92}. As far as we know, the second equality in \eqref{23.1.12} is a new wrinkle in this context (inspired by work of Rellich \cite{Re43} and \cite{BG85}, \cite{GP84}) and constitutes our the main contribution to this circle of ideas. 

We also note that Niessen and Zettl \cite{NZ92} rely on the possibility of transforming a singular, nonoscillatory Sturm--Liouville problem at a limit circle endpoint into a regular one at that endpoint (see also \cite{AF94}). We avoid this technique in this paper and instead employ most elementary analytical manipulations only. 

We could have written a short note and just focused on Theorem \ref{23.t7.3.11}. Instead, we decided to make this more readable and somewhat self-contained by introducing nearly all basic objects entering into our considerations. Section \ref{23.s2} recalls the case of regular Sturm--Liouville operators and their self-adjoint boundary conditions, Section \ref{23.s3} then provides the same in the singular context. Section \ref{23.s4} then recalls principal and nonprincipal solutions in the case where $T_{min}$ is bounded from below (and, hence, so are all its self-adjoint extensions), and proves our main result, Theorem \ref{23.t7.3.11}. In Section \ref{23.s5}, we touch on Weyl--Titchmarsh functions and their connections to the generalized boundary values introduced in Section \ref{23.s4}.  Finally, Section \ref{23.s6} illustrates these concepts with the help of a number of representative examples such as the Bessel, Legendre, and Laguerre (resp., Kummer) operators.

Finally, we comment on some of the basic notation used throughout this paper.  If $T$ is a linear operator mapping (a subspace of) a Hilbert space into another, then $\dom(T)$ and $\ker(T)$ denote the domain and kernel (i.e., null space) of $T$. The spectrum and resolvent set of a closed linear operator in a Hilbert space will be denoted by $\sigma(\cdot)$ and $\rho(\cdot)$, respectively.

\section{Self-Adjoint Regular Sturm--Liouville Operators} \lb{23.s2}

To set the stage and introduce the notation for our principal Section \ref{23.s3}, we now briefly recall the basic facts on 
regular Sturm--Liouville operators and their self-adjoint boundary conditions. Everything in this section is standard and well-known, hence we just refer to some of the standard monographs on this subject, such as, 
\cite[Sect.~6.3]{BHS19}, \cite[Sect.~II.5]{JR76}, 
\cite[Ch.~V]{Na68}, \cite[Sect.~8.4]{We80}, \cite[Sect.~13.2]{We03}, \cite[Ch.~4]{Ze05}.

Throughout this section we make the following assumptions:

\begin{hypothesis} \lb{23.h3.1}
Let $(a,b) \subset \bbR$ be a finite interval and suppose that $p,q,r$ are $($Lebesgue\,$)$ measurable functions on $(a,b)$  
such that the following items $(i)$--$(iii)$ hold: \\[1mm] 
$(i)$ $r > 0$ a.e.~on $(a,b)$, $r \in L^1((a,b);dx)$. \\[1mm]
$(ii)$ $p > 0$ a.e.~on $(a,b)$, $1/p \in L^1((a,b);dx)$. \\[1mm]
$(iii)$ $q$ is real-valued a.e.~on $(a,b)$, $q \in L^1((a,b);dx)$.  
\end{hypothesis}

Given Hypothesis \ref{23.h3.1}, we now study Sturm--Liouville operators associated with the general, 
three-coefficient differential expression $\tau$ of the type,
\begin{equation}
\tau=\f{1}{r(x)}\left[-\f{d}{dx}p(x)\f{d}{dx} + q(x)\right] \, \text{ for a.e.~$x\in(a,b) \subseteq \bbR$.} 
   \lb{23.2.1}
\end{equation} 

We start with the notion of minimal and maximal $\Lr$-realizations associated with the regular differential 
expression $\tau$ on the finite interval $(a,b) \subset \bbR$.  Here, and elsewhere throughout this manuscript, the inner product in $\Lr$ is defined by
\begin{equation}
(f,g)_{\Lr} = \int_a^b \ol{f(x)}g(x)r(x)\, dx,\quad f,g\in \Lr.
\end{equation}

\begin{definition} \lb{23.d3.1}
Assume Hypothesis \ref{23.h3.1}. \\[1mm] 
$(i)$ Then the differential expression $\tau$ of the form \eqref{23.2.1} on the finite interval 
$(a,b) \subset \bbR$ is called {\it regular on} $[a,b]$. \\[1mm] 
$(ii)$ The \textit{maximal operator} $T_{max}$ in $\Lr$ associated with $\tau$ is defined by 
\begin{align}
&T_{max} f = \tau f,     \no
\\
& f \in \dom(T_{max})=\big\{g\in\Lr  \, \big| \,  g,g^{[1]}\in\AC;     \lb{23.3.1} \\
& \hspace*{6cm} \tau g\in\Lr\big\}.   \no
\end{align}
The \textit{minimal operator} $T_{min}$ in $\Lr$ associated with $\tau$ is defined by 
\begin{align}
&T_{min} f = \tau f, \no
\\
& f \in \dom(T_{min})=\big\{g\in\Lr  \, \big| \,  g,g^{[1]}\in\AC;     \lb{23.3.2} \\ 
&\hspace*{3cm} g(a)=g^{[1]}(a)=g(b)=g^{[1]}(b)=0; \; \tau g\in\Lr\big\}.  \no 
\end{align}
\end{definition}

The following is then a fundamental result:

\begin{theorem} \lb{23.t3.4}
Assume Hypothesis \ref{23.h3.1} so that $\tau$ is regular on $[a,b]$. Then $T_{min}$ is a densely defined, closed operator in $\Lr$. Moreover, $T_{max}$ is densely defined and closed in $\Lr$, and 
\begin{align}
T_{min}^* = T_{max}, \quad T_{min} = T_{max}^*.
\end{align}
In addition, $T_{min} \subset T_{max} = T_{min}^*$, and hence $T_{min}$ is symmetric. Finally, 
$T_{max}$ is not symmetric. 
\end{theorem}

A first glimpse at self-adjoint extensions of $T_{min}$ is provided by the following result.

\begin{lemma} \lb{23.l3.8}
Assume Hypothesis \ref{23.h3.1} so that $\tau$ is regular on $[a,b]$. Then an extension $\wti T$ of \ $T_{min}$ is self-adjoint if and only if
\begin{align}
& \wti T f = \tau f,   \no \\
& f \in \dom\big(\wti T\big)=\big\{g\in\dom(T_{max})  \, \big| \,  W\big(\ol{f},g\big)\vl_a^b=0 \text{ for all } f\in\dom(\wti T)\big\}    \lb{23.3.22}  \\
& \hspace*{1.93cm} 
=\big\{g\in\dom(T_{max})  \, \big| \,  W\big(f,g\big)\vl_a^b=0 \text{ for all } f\in\dom(\wti T)\big\}.   \no
\end{align} 
\end{lemma}

Here the Wronskian of $f$ and $g$, for $f,g\in\ACl$, is defined by
\begin{equation}
W(f,g)(x) = f(x)g^{[1]}(x) - f^{[1]}(x)g(x),    \lb{23.2.3.1} \no \\
\end{equation}
with 
\begin{equation}
y^{[1]}(x) = p(x) y'(x),
\end{equation}
denoting the first quasi-derivative of a function $y\in AC_{loc}((a,b))$.

The next theorem describes the self-adjoint extensions of $T_{min}$ in more detail.

\begin{theorem} \lb{23.t3.9}
Assume Hypothesis \ref{23.h3.1} so that $\tau$ is regular on $[a,b]$. Then $T_{A,B}$ is a self-adjoint extension of $T_{min}$ if and only if there exist $2\times2$ matrices $A$ and $B$ $($with complex-valued entries$)$ satisfying
\begin{align}
\rank(A \quad B) = 2, \quad AJA^* = BJB^*, \quad J=\begin{pmatrix}0&-1\\1&0\end{pmatrix},  \lb{23.3.24}
\end{align}
with $T_{A,B}$ given by 
\begin{align}
\begin{split} 
& T_{A,B} f = \tau f,     \\
& f \in \dom(T_{A,B})=\left\{g\in\dom(T_{max}) \, \bigg| \, A\begin{pmatrix}g(a)\\g^{[1]}(a)\end{pmatrix}=B\begin{pmatrix}g(b)\\g^{[1]}(b)\end{pmatrix} \right\}. \lb{23.3.25}
\end{split} 
\end{align} 
\end{theorem}

Concluding this section, we turn to the important special cases of separated and coupled boundary conditions which together describe all self-adjoint extensions of $T_{min}$.

\begin{theorem} \lb{23.t3.10}
Assume Hypothesis \ref{23.h3.1} so that $\tau$ is regular on $[a,b]$. Then the following items $(i)$--$(iii)$ hold: \\[1mm] 
$(i)$ All self-adjoint extensions $T_{\al,\be}$ of $T_{min}$ with separated boundary conditions are of the form
\begin{align}
& T_{\al,\be} f = \tau f, \quad \al,\be\in[0,\pi),   \no \\
& f \in \dom(T_{\al,\be})=\big\{g\in\dom(T_{max}) \, \big| \, g(a)\cos(\al)+g^{[1]}(a)\sin(\al)=0;   \lb{23.3.52} \\ 
& \hspace*{5.6cm} g(b)\cos(\be)+g^{[1]}(b)\sin(\be) = 0 \big\}.    \no 
\end{align}
Special cases: $\al=0$, $g(a)=0$ is called the Dirichlet boundary condition at $a$; $\al=\f\pi2$,  
$g^{[1]}(a)=0$ is called the Neumann boundary condition at $a$ $($analogous facts hold at the 
endpoint $b$$)$. \\[1mm]
$(ii)$ All self-adjoint extensions $T_{\varphi,R}$ of $T_{min}$ with coupled boundary conditions are of the type
\begin{align}
\begin{split} 
& T_{\varphi,R} f = \tau f,    \\
& f \in \dom(T_{\varphi,R})=\bigg\{g\in\dom(T_{max}) \, \bigg| \begin{pmatrix}g(b)\\g^{[1]}(b)\end{pmatrix} 
= e^{i\varphi}R \begin{pmatrix}
g(a)\\g^{[1]}(a)\end{pmatrix} \bigg\}, \lb{23.3.53}
\end{split}
\end{align}
where $\varphi\in[0,2\pi)$, and $R$ is a real $2\times2$ matrix with $\det(R)=1$ 
$($i.e., $R \in SL(2,\bbR)$$)$.
Special cases: $\varphi = 0$, $R=I_2$, $g(b)=g(a)$, $g^{[1]}(b)=g^{[1]}(a)$ are called {\it periodic boundary conditions}; 
similarly, $\varphi = \pi$, $R=I_2$, $g(b)=-g(a)$, $g^{[1]}(b)=-g^{[1]}(a)$ are called {\it antiperiodic boundary conditions}. 
\\[1mm] 
$(iii)$ Every self-adjoint extension of $T_{min}$ is either of type $(i)$ $($i.e., separated\,$)$ or of type 
$(ii)$ $($i.e., coupled\,$)$.
\end{theorem}

For interesting extensions of this circle of ideas to other types of boundary conditions, see, for instance, \cite{Ga05} and the references cited therein.

\section{Self-Adjoint Singular Sturm--Liouville Operators} \lb{23.s3}

In this section we first recall the basics of singular Sturm--Liouville operators and add a new wrinkle to the existing theory at the end of this section. Again, the material described up to, and including,  Lemma \ref{23.l2.7} is standard and can be found, for instance, in \cite[Chs.~8, 9]{CL85}, \cite[Sects~13.6, 13.9, 13.0]{DS88}, 
\cite[Ch.~III]{JR76}, \cite[Ch.~V]{Na68}, \cite{NZ92}, \cite[Ch.~6]{Pe88}, \cite[Ch.~9]{Te14}, \cite[Sect.~8.3]{We80}, \cite[Ch.~13]{We03}, \cite[Chs.~4, 6--8]{Ze05}.

Throughout this section we make the following assumptions:

\begin{hypothesis} \lb{23.h2.1}
Let $(a,b) \subseteq \bbR$ and suppose that $p,q,r$ are $($Lebesgue\,$)$ measurable functions on $(a,b)$ 
such that the following items $(i)$--$(iii)$ hold: \\[1mm] 
$(i)$ \hspace*{1.1mm} $r>0$ a.e.~on $(a,b)$, $r\in\Ll$. \\[1mm] 
$(ii)$ \hspace*{.1mm} $p>0$ a.e.~on $(a,b)$, $1/p \in\Ll$. \\[1mm] 
$(iii)$ $q$ is real-valued a.e.~on $(a,b)$, $q\in\Ll$. 
\end{hypothesis}

Given Hypothesis \ref{23.h2.1}, we again study Sturm--Liouville operators associated with the general, 
three-coefficient differential expression $\tau$ in \eqref{23.2.1}. 

\begin{definition} \lb{23.d4.1}
Assume Hypothesis \ref{23.h2.1}. Given $\tau$ as in \eqref{23.2.1}, the \textit{maximal operator} $T_{max}$ in $\Lr$ associated with $\tau$ is defined by
\begin{align}
&T_{max} f = \tau f,    \no
\\
& f \in \dom(T_{max})=\big\{g\in\Lr \, \big| \,g,g^{[1]}\in\ACl;   \lb{23.4.2.1} \\ 
& \hspace*{6.3cm}  \tau g\in\Lr\big\}.   \no
\end{align}
The \textit{minimal operator} $T_{min,0} $ in $\Lr$ associated with $\tau$ is defined by 
\begin{align}
&T_{min,0}  f = \tau f,   \no
\\
&f \in \dom (T_{min,0})=\big\{g\in\Lr \, \big| \, g,g^{[1]}\in\ACl;   \lb{23.4.2.2}
\\
&\hspace*{3.25cm} \supp \, (g)\subset(a,b) \text{ is compact; } \tau g\in\Lr\big\}.   \no
\end{align}

One can prove that $T_{min,0} $ is closable, and one then defines $T_{min}$ 
as the closure of $T_{min,0} $.
\end{definition}

The Green identity and one of its consequences are discussed next.

\begin{lemma} \lb{23.l4.2}
Assume Hypothesis \ref{23.h2.1}, then the following items $(i)$--$(iii)$ hold: \\[1mm] 
$(i)$ Suppose $f, f^{[1]}, g, g^{[1]}\in\ACl$ and $[\al,\be]\subset(a,b)$, then the Lagrange or Green identity reads,
\begin{align}
\int_\al^\be dx\, r(x)\big(\ol{(\tau f)(x)}g(x)-(\tau g)(x)\ol{f(x)}\big) = W(\ol{f},g)(\be) - W(\ol{f},g)(\al). \lb{23.4.4}
\end{align}
$(ii)$ If $f,g\in\dom(T_{max})$, then the limits $W(\ol{f},g)(a)=\lim_{x\downarrow a}W(\ol{f},g)(x)$ and $W(\ol{f},g)(b)=\lim_{x\uparrow b}W(\ol{f},g)(x)$ exist and
\begin{align}
(T_{max} f,g)_{L^2((a,b);rdx)}-(f,T_{max} g)_{L^2((a,b);rdx)} = W(\ol{f},g)(b) - W(\ol{f},g)(a). \lb{23.4.5}
\end{align}
\end{lemma}

One can prove the following basic fact:

\begin{theorem} \lb{23.t4.6} 
Assume Hypothesis \ref{23.h2.1}. Then 
\begin{equation} 
(T_{min,0})^* = T_{max}, 
\end{equation} 
and hence $T_{max}$ is closed and $T_{min}=\ol{T_{min,0} }$ is given by
\begin{align}
&T_{min} f = \tau f, \no
\\
&f \in \dom(T_{min})=\big\{g\in\Lr  \, \big| \,  g,g^{[1]}\in\ACl;     \\
& \qquad \text{for all } h\in\dom(T_{max}), \, W(h,g)(a)=0=W(h,g)(b); \, \tau g\in\Lr\big\} \no
\\
& \quad =\big\{g\in\dom(T_{max})  \, \big| \, W(h,g)(a)=0=W(h,g)(b) \, 
\text{for all } h\in\dom(T_{max}) \big\}.   \no 
\end{align}
Moreover, $T_{min,0} $ is essentially self-adjoint if and only if\; $T_{max}$ is symmetric, and then $\ol{T_{min,0} }=T_{min}=T_{max}$.
\end{theorem}

Regarding self-adjoint extensions of $T_{min}$ one has the following first result.

\begin{theorem} \lb{23.t4.7} 
Assume Hypothesis \ref{23.h2.1}. 
An extension $\wti T$ of $T_{min,0} $ or of $T_{min}=\ol{T_{min,0} }$ is self-adjoint if and only if 
\begin{align}
& \wti T f = \tau f,    \\
& f \in \dom\big(\wti T\big) = \big\{g\in\dom(T_{max})  \, \big| \,  
W(f,g)(a)=W(f,g)(b) \text{ for all } f\in\dom\big(\wti T\big)\big\}.   \no 
\end{align}
\end{theorem}

The celebrated Weyl alternative then can be stated as follows:

\begin{theorem}[Weyl's Alternative] \lb{t4.9} ${}$ \\
Assume Hypothesis \ref{23.h2.1}. Then the following alternative holds: \\[1mm] 
$(i)$ For every $z\in\bbC$, all solutions $u$ of $(\tau-z)u=0$ are in $\Lr$ near $b$ 
$($resp., near $a$$)$. \\[1mm] 
$(ii)$  For every $z\in\bbC$, there exists at least one solution $u$ of $(\tau-z)u=0$ which is not in $\Lr$ near $b$ $($resp., near $a$$)$. In this case, for each $z\in\bbC\bs\bbR$, there exists precisely one solution $u_b$ $($resp., $u_a$$)$ of $(\tau-z)u=0$ $($up to constant multiples$)$ which lies in $\Lr$ near $b$ $($resp., near $a$$)$. 
\end{theorem}

This yields the limit circle/limit point classification of $\tau$ at an interval endpoint as follows. 

\begin{definition} \lb{23.d4.10} 
Assume Hypothesis \ref{23.h2.1}. \\[1mm]  
In case $(i)$ in Theorem \ref{t4.9}, $\tau$ is said to be in the \textit{limit circle case} at $b$ $($resp., at $a$$)$. $($Frequently, $\tau$ is then called \textit{quasi-regular} at $b$ $($resp., $a$$)$.$)$
\\[1mm] 
In case $(ii)$ in Theorem \ref{t4.9}, $\tau$ is said to be in the \textit{limit point case} at $b$ $($resp., at $a$$)$. \\[1mm]
If $\tau$ is in the limit circle case at $a$ and $b$ then $\tau$ is also called \textit{quasi-regular} on $(a,b)$. 
\end{definition}

The next result links self-adjointness of $T_{min}$ (resp., $T_{max}$) and the limit point property of $\tau$ at both endpoints:

\begin{theorem} \lb{23.t4.11}
Assume Hypothesis~\ref{23.h2.1}, then the following items $(i)$ and $(ii)$ hold: \\[1mm] 
$(i)$ If $\tau$ is in the limit point case at $a$ $($resp., $b$$)$, then 
\begin{equation} 
W(f,g)(a)=0 \, \text{$($resp., $W(f,g)(b)=0$$)$ for all $f,g\in\dom(T_{max})$.} 
\end{equation} 
$(ii)$ Let $T_{min}=\ol{T_{min,0} }$. Then
\begin{align}
\begin{split}
n_\pm(T_{min}) &= \dim(\ker(T_{max} \mp i I))    \\
& = \begin{cases}
2 & \text{if $\tau$ is in the limit circle case at $a$ and $b$,}\\
1 & \text{if $\tau$ is in the limit circle case at $a$} \\
& \text{and in the limit point case at $b$, or vice versa,}\\
0 & \text{if $\tau$ is in the limit point case at $a$ and $b$}.
\end{cases}
\end{split} 
\end{align}
In particular, $T_{min} = T_{max}$ is self-adjoint if and only if $\tau$ is in the limit point case at $a$ and $b$. 
\end{theorem}

We continue with a discussion of self-adjoint extensions in the quasi-regular case, that is, the case where $\tau$ is in the limit circle case at $a$ and $b$.

\begin{theorem} \lb{23.t4.12}
Assume Hypothesis~\ref{23.h2.1}, suppose that $\tau$ is in the limit circle case at $a$ and $b$ $($i.e., $\tau$ is 
quasi-regular on $(a,b)$$)$, and let $\la\in\bbR$. 
Then $T$ is a self-adjoint extension of $T_{min}$ if and only if there exist functions 
$v,w:(a,b)\to\bbC$ such that \\[1mm] 
$(\alpha)$ $v,w$ are solutions of $(\tau-\la)u=0$ in a neighborhood of $a$ and $b$. 
\\[1mm] 
$(\beta)$ $v,w$ are linearly independent$\mod(\dom(T_{min}))$ $($i.e., no 
nontrivial linear combination of $v$ and $w$ vanishes simultaneously near $a$ and $b$$)$. \\[1mm] 
$(\gamma)$ One has 
\begin{align} 
& W(\ol{v},w)\vl_a^b=W(\ol{v},v)\vl_a^b=W(\ol{w},w)\vl_a^b=0, \\ 
& \, \dom(T)=\Big\{g\in\dom(T_{max})  \, \Big| \,  W(v,g)\vl_a^b=W(w,g)\vl_a^b=0\Big\}. 
\end{align} 
Instead of $v,w$, one can also use any functions $h, k$ which are linearly 
independent $\mod(\dom(T_{min}))$ $($i.e., for no nontrivial linear combination $u=c_1h+c_2 k$ one 
has $W(u,g)\vl_a^b=0$ for all $g\in\dom(T_{max})$$)$ as long as 
$W(\ol{h}, k)\vl_a^b=W(\ol{h}, h)\vl_a^b=W(\ol{k}, k)\vl_a^b=0$. 
\end{theorem}

The next result discusses all self-adjoint extensions of $T_{min}$ with separated boundary conditions. This is no restriction as long as $\tau$ is in the limit point case at one of the endpoints $a$ or $b$.

\begin{theorem} \lb{23.t4.13}
Assume Hypothesis \ref{23.h2.1} and let $\la\in\bbR$. Then $T$ is a self-adjoint extension of $T_{min}$ with separated boundary conditions if and only if there are nontrivial real-valued solutions $v,w$ of $(\tau-\la)u=0$ with
\begin{align}
& T f = \tau f,     \lb{23.4.5.6} \\
& f \in \dom(T)=\{g\in\dom(T_{max})  \, | \,  W(v,g)(a)=0 \text{ if $\tau$ is in the limit circle case}    \no \\
& \hspace*{3.3cm} \text{ at $a$; $W(w,g)(b)=0$ if $\tau$ is in the limit circle case at $b$}\}   \no 
\end{align}
$($with a boundary condition omitted whenever $\tau$ is in the limit point case at $a$ and/or $b$$)$. We 
denote such a self-adjoint extension $T$ by $T_{v,w}$ $($and we drop the subscript $v$ or $w$ if 
$\tau$ is in the limit point case at $a$ or $b$$)$. 
\end{theorem}

Restricting the solutions $v, w$ employed in Theorems \ref{23.t4.12} and \ref{23.t4.13} appropriately, one can derive analogs of Theorem \ref{23.t3.9}  and \ref{23.t3.10} in the regular case in the present singular, quasi-regular setting as follows:

\begin{theorem} \lb{23.t4.15}
Assume Hypothesis \ref{23.h2.1} and that $\tau$ is in the limit circle case at $a$ and $b$ $($i.e., $\tau$ is quasi-regular 
on $(a,b)$$)$. In addition, assume that 
$v_j \in \dom(T_{max})$, $j=1,2$, satisfy 
\begin{equation}
W(\ol{v_1}, v_2)(a) = W(\ol{v_1}, v_2)(b) = 1, \quad W(\ol{v_j}, v_j)(a) = W(\ol{v_j}, v_j)(b) = 0, \; j= 1,2.  
\end{equation}
$($E.g., real-valued solutions $v_j$, $j=1,2$, of $(\tau - \lambda) u = 0$ with $\lambda \in \bbR$, such that 
$W(v_1,v_2) = 1$.$)$ For $g\in\dom(T_{max})$ we introduce the generalized boundary values 
\begin{align}
\begin{split} 
\wti g_1(a) &= - W(v_2, g)(a), \quad \wti g_1(b) = - W(v_2, g)(b),    \\
\wti g_2(a) &= W(v_1, g)(a), \quad \;\,\,\, \wti g_2(b) = W(v_1, g)(b).   \lb{23.4.5.16}
\end{split} 
\end{align}
Then the following items $(i)$--$(iv)$ hold: \\[1mm] 
$(i)$ $T_{A,B}$ is a self-adjoint extension of $T_{min}$ if and only if there exist $2\times2$ matrices $A$ and $B$ 
$($with complex-valued entries$)$ satisfying
\begin{align}
\rank(A \quad B) = 2, \quad AJA^* = BJB^*, \quad J=\begin{pmatrix}0&-1\\1&0\end{pmatrix},  \lb{23.3.24A}
\end{align}
with $T_{A,B}$ given by 
\begin{align}
\begin{split} 
& T_{A,B} f = \tau f,     \\
& f \in \dom(T_{A,B})=\left\{g\in\dom(T_{max}) \, \bigg| \, A\begin{pmatrix} \wti g_1(a)\\ \wti g_2(a)\end{pmatrix}=B\begin{pmatrix} \wti g_1(b)\\ \wti g_2(b)\end{pmatrix} \right\}. \lb{23.3.25A}
\end{split} 
\end{align}
In addition, items $(ii)$ and $(iii)$ in Theorem \ref{23.t4.12} apply to $T_{A,B}$. \\[1mm] 
$(ii)$ All self-adjoint extensions $T_{\al,\be}$ of $T_{min}$ with separated boundary conditions are of the form
\begin{align}
& T_{\al,\be} f = \tau f, \quad \al,\be\in[0,\pi),   \no \\
& f \in \dom(T_{\al,\be})=\big\{g\in\dom(T_{max}) \, \big| \, \wti g_1(a)\cos(\al)+ \wti g_2(a)\sin(\al)=0;   \lb{23.3.52A} \\ 
& \hspace*{5.5cm} \, \wti g_1(b)\cos(\be)+ \wti g_2(b)\sin(\be) = 0 \big\}.    \no 
\end{align}
$(iii)$ All self-adjoint extensions $T_{\varphi,R}$ of $T_{min}$ with coupled boundary conditions are of the type
\begin{align}
\begin{split} 
& T_{\varphi,R} f = \tau f,    \\
& f \in \dom(T_{\varphi,R})=\bigg\{g\in\dom(T_{max}) \, \bigg| \begin{pmatrix} \wti g_1(b)\\ \wti g_2(b)\end{pmatrix} 
= e^{i\varphi}R \begin{pmatrix}
\wti g_1(a)\\ \wti g_2(a)\end{pmatrix} \bigg\}, \lb{23.3.53A}
\end{split}
\end{align}
where $\varphi\in[0,2\pi)$, and $R$ is a real $2\times2$ matrix with $\det(R)=1$ 
$($i.e., $R \in SL(2,\bbR)$$)$.  \\[1mm] 
$(iv)$ Every self-adjoint extension of $T_{min}$ is either of type $(ii)$ $($i.e., separated\,$)$ or of type 
$(iii)$ $($i.e., coupled\,$)$.
\end{theorem}

\begin{remark} \lb{23.r4.16}
$(i)$ If $\tau$ is in the limit point case at one endpoint, say, at the endpoint $b$, one omits the corresponding boundary condition involving $\beta \in [0, \pi)$ at $b$ in \eqref{23.3.52A} to obtain all self-adjoint extensions $T_{\alpha}$ of 
$T_{min}$, indexed by $\alpha \in [0, \pi)$. (In this case item $(iii)$ in Theorem \ref{23.t4.15} is vacuous.) In the case where $\tau$ is in the limit point case at both endpoints, all boundary values and boundary conditions become superfluous as in this case $T_{min} = T_{max}$ is self-adjoint. \\[1mm] 
$(ii)$ In the special case where $\tau$ is regular on the finite interval $[a,b]$, choose $v_j \in \dom(T_{max})$, $j=1,2$, 
such that 
\begin{align}
v_1(x) = \begin{cases} \theta_0(\lambda,x,a), & \text{for $x$ near a}, \\
\theta_0(\lambda,x,b), & \text{for $x$ near b},  \end{cases}   \quad 
v_2(x) = \begin{cases} \phi_0(\lambda,x,a), & \text{for $x$ near a}, \\
\phi_0(\lambda,x,b), & \text{for $x$ near b},  \end{cases}   \lb{23.4.5.21}
\end{align} 
where $\phi_0(\lambda,\, \cdot \,,d)$, $\theta_0(\lambda,\, \cdot \,,d)$, $d \in \{a,b\}$, are real-valued solutions of $(\tau - \lambda) u = 0$, $\lambda \in \bbR$, satisfying the boundary conditions 
\begin{align}
\begin{split} 
& \phi_0(\lambda,a,a) = \theta_0^{[1]}(\lambda,a,a) = 0, \quad \theta_0(\lambda,a,a) = \phi_0^{[1]}(\lambda,a,a) = 1, \\ 
& \phi_0(\lambda,b,b) = \theta_0^{[1]}(\lambda,b,b) = 0, \quad \; \theta_0(\lambda,b,b) = \phi_0^{[1]}(\lambda,b,b) = 1. 
\lb{23.4.5.22}
\end{split} 
\end{align} 
Then one verifies that
\begin{align}
\wti g_1 (a) = g(a), \quad \wti g_1 (b) = g(b), \quad \wti g_2 (a) = g^{[1]}(a), \quad \wti g_2 (b) = g^{[1]}(b),   \lb{23.4.5.23} 
\end{align}
and hence Theorem \ref{23.t4.15} in the special regular case recovers Theorems \ref{23.t3.9} and \ref{23.t3.10}. \\[1mm]
$(iii)$ An explicit calculation demonstrates that for $g, h \in \dom(T_{max})$,
\begin{equation}
\wti g_1(d) \wti h_2(d) - \wti g_2(d) \wti h_1(d) = W(g,h)(d), \quad d \in \{a,b\},   \lb{23.4.5.24}
\end{equation} 
interpreted in the sense that either side in \eqref{23.4.5.24} has a finite limit as $d \downarrow a$ and $d \uparrow b$. 
Of course, for \eqref{23.4.5.24} to hold at $d \in \{a,b\}$, it suffices that $g$ and $h$ lie locally in $\dom(T_{max})$ near $x=d$. \\[1mm]
$(iv)$ Clearly, $\wti g_1, \wti g_2$ depend on the choice of $v_j$, $j=1,2$, and a more precise notation would indicate this as $\wti g_{1,v_2}, \wti g_{2,v_1}$, etc. \hfill $\diamond$
\end{remark} 

In the special case where $T_{min}$ is bounded from below, one can further analyze the generalized boundary values  
\eqref{23.4.5.16} in the singular context by invoking principal and nonprincipal solutions of $\tau u = \lambda u$ for appropriate $\lambda \in \bbR$. This leads to natural analogs of \eqref{23.4.5.23} also in the singular case, and we will turn to this topic in our next section.

\section{Generalized Boundary Values in the Semibounded Case} \lb{23.s4}

In this section we characterize generalized boundary values in the case where $T_{min}$ is bounded from below with the help of principal and nonprincipal solutions of $\tau u = \lambda u$ for appropriate $\lambda \in \bbR$, and recall how they can be used to characterize all self-adjoint extensions of $T_{min}$.

We start by reviewing some oscillation theory with particular emphasis on principal and nonprincipal solutions, a notion originally due to Leighton and Morse \cite{LM36} (see also Rellich \cite{Re43}, \cite{Re51} and Hartman and Wintner \cite[Appendix]{HW55}). Our outline below follows \cite{CGN16}, \cite[Sects~13.6, 13.9, 13.0]{DS88}, 
\cite[Ch.~XI]{Ha02}, \cite{NZ92}, \cite[Chs.~4, 6--8]{Ze05}. 

\begin{definition} \lb{23.d2.4}
Assume Hypothesis \ref{23.h2.1}. \\[1mm] 
$(i)$ Fix $c\in (a,b)$ and $\lambda\in\bbR$. Then $\tau - \lam$ is
called {\it nonoscillatory} at $a$ $($resp., $b$$)$, 
if every real-valued solution $u(\lambda,\dott)$ of 
$\tau u = \lambda u$ has finitely many
zeros in $(a,c)$ $($resp., $(c,b)$$)$. Otherwise, $\tau - \lam$ is called {\it oscillatory}
at $a$ $($resp., $b$$)$. \\[1mm] 
$(ii)$ Let $\lambda_0 \in \bbR$. Then $T_{min}$ is called bounded from below by $\lambda_0$, 
and one writes $T_{min} \geq \lambda_0 I$, if 
\begin{equation} 
(u, [T_{min} - \lambda_0 I]u)_{L^2((a,b);rdx)}\geq 0, \quad u \in \dom(T_{min}).
\end{equation}
\end{definition}

The following is a key result. 

\begin{theorem} \lb{23.t2.5} 
Assume Hypothesis \ref{23.h2.1}. Then the following items $(i)$--$(iii)$ are
equivalent: \\[1mm] 
$(i)$ $T_{min}$ $($and hence any symmetric extension of $T_{min})$
is bounded from below. \\[1mm] 
$(ii)$ There exists a $\nu_0\in\bbR$ such that for all $\lambda < \nu_0$, $\tau - \lam$ is
nonoscillatory at $a$ and $b$. \\[1mm]
$(iii)$ For fixed $c, d \in (a,b)$, $c \leq d$, there exists a $\nu_0\in\bbR$ such that for all
$\lambda<\nu_0$, $\tau u = \lambda u$ has $($real-valued\,$)$ nonvanishing solutions
$u_a(\lambda,\dott) \neq 0$,
$\hatt u_a(\lambda,\dott) \neq 0$ in a neighborhood $(a,c]$ of $a$, and $($real-valued\,$)$ nonvanishing solutions
$u_b(\lambda,\dott) \neq 0$, $\hatt u_b(\lambda,\dott) \neq 0$ in a neighborhood $[d,b)$ of
$b$, such that 
\begin{align}
&W(\hatt u_a (\lambda,\dott),u_a (\lambda,\dott)) = 1,
\quad u_a (\lambda,x)=\oh(\hatt u_a (\lambda,x))
\text{ as $x\downarrow a$,} \lb{23.2.9} \\
&W(\hatt u_b (\lambda,\dott),u_b (\lambda,\dott))\, = 1,
\quad u_b (\lambda,x)\,=\oh(\hatt u_b (\lambda,x))
\text{ as $x\uparrow b$,} \lb{23.2.9a} \\
&\int_a^c dx \, p(x)^{-1}u_a(\lambda,x)^{-2}=\int_d^b dx \, 
p(x)^{-1}u_b(\lambda,x)^{-2}=\infty,  \lb{23.2.10} \\
&\int_a^c dx \, p(x)^{-1}{\hatt u_a(\lambda,x)}^{-2}<\infty, \quad 
\int_d^b dx \, p(x)^{-1}{\hatt u_b(\lambda,x)}^{-2}<\infty. \lb{23.2.11}
\end{align}
\end{theorem}

\begin{definition} \lb{23.d2.6}
Assume Hypothesis \ref{23.h2.1}, suppose that $T_{min}$ is bounded from below, and let 
$\lambda\in\bbR$. Then $u_a(\lambda,\dott)$ $($resp., $u_b(\lambda,\dott)$$)$ in Theorem
\ref{23.t2.5}\,$(iii)$ is called a {\it principal} $($or {\it minimal}\,$)$
solution of $\tau u=\lambda u$ at $a$ $($resp., $b$$)$. A real-valued solution 
$\wti u_a(\lambda,\dott)$ $($resp., $\wti u_b(\lambda,\dott)$$)$ of $\tau
u=\lambda u$ linearly independent of $u_a(\lambda,\dott)$ $($resp.,
$u_b(\lambda,\dott)$$)$ is called {\it nonprincipal} at $a$ $($resp., $b$$)$.
\end{definition}

Principal and nonprincipal solutions are well-defined due to Lemma \ref{23.l2.7}\,$(i)$ below. 

\begin{lemma} \lb{23.l2.7} Assume Hypothesis \ref{23.h2.1} and suppose that $T_{min}$ is bounded 
from below. \\[1mm]
$(i)$ $u_a(\lambda,\dott)$ and $u_b(\lambda,\dott)$ in Theorem
\ref{23.t2.5}\,$(iii)$ are unique up to $($nonvanishing\,$)$ real constant multiples. Moreover,
$u_a(\lambda,\dott)$ and $u_b(\lambda,\dott)$ are minimal solutions of
$\tau u=\lambda u$ in the sense that 
\begin{align}
u(\lambda,x)^{-1} u_a(\lambda,x)&=\oh(1) \text{ as $x\downarrow a$,} 
\lb{23.2.12} \\ 
u(\lambda,x)^{-1} u_b(\lambda,x)&=\oh(1) \text{ as $x\uparrow b$,} \lb{23.2.12a}
\end{align}
for any other solution $u(\lambda,\dott)$ of $\tau u=\lambda u$
$($which is nonvanishing near $a$, resp., $b$$)$ with
$W(u_a(\lambda,\dott),u(\lambda,\dott))\neq 0$, respectively, 
$W(u_b(\lambda,\dott),u(\lambda,\dott))\neq 0$. \\[1mm]
$(ii)$ Let $u(\lambda,\dott) \neq 0$ be any nonvanishing solution of $\tau u=\lambda
u$ near $a$ $($resp., $b$$)$. Then for $c_1>a$ $($resp., $c_2<b$$)$
sufficiently close to $a$ $($resp., $b$$)$, 
\begin{align}
& \hatt u_a(\lambda,x)=u(\lambda,x)\int_x^{c_1}dx' \, 
p(x')^{-1}u(\lambda,x')^{-2} \lb{23.2.13} \\
& \quad \bigg(\text{resp., }\hatt u_b(\lambda,x)=u(\lambda,x)\int^x_{c_2}dx' \, 
p(x')^{-1}u(\lambda,x')^{-2}\bigg) \lb{23.2.14} 
\end{align} 
is a nonprincipal solution of $\tau u=\lambda u$ at $a$ $($resp.,
$b$$)$. If $\hatt u_a(\lambda,\dott)$ $($resp., $\hatt
u_b(\lambda,\dott))$ is a nonprincipal solution of $\tau u=\lambda u$
at $a$ $($resp., $b$$)$ then 
\begin{align}
& u_a(\lambda,x)=\hatt u_a(\lambda,x)\int_a^{x}dx' \, 
p(x')^{-1}{\hatt u_a(\lambda,x')}^{-2} \lb{23.2.15} \\
& \quad \bigg(\text{resp., } u_b(\lambda,x)=\hatt u_b(\lambda,x)\int^b_{x}dx' \, 
p(x')^{-1}{\hatt u_b(\lambda,x')}^{-2}\bigg) \lb{23.2.16} 
\end{align} 
is a principal solution of $\tau u=\lambda u$ at $a$ $($resp., $b$$)$. 
\end{lemma}

Next, we revisit  in Theorem \ref{23.t4.15} how the generalized boundary values are utilized in the description of all self-adjoint extensions of $T_{min}$. In particular, assuming $T_{min} \geq \lambda_0 I$, $\lambda_0 \in \bbR$, to be bounded from below, naturally leads to invoking principal and nonprincipal solutions of $\tau u = \lambda_0 u$. 

Assuming Hypothesis \ref{23.h2.1} and $f\in\Ll$, and given a fundamental system of solutions $y_1(z,\dott), y_2(z,\dott)$ 
of the homogeneous second order equation $\tau y = z y$, one recalls (via the variation of constants formula) that the general solution of the corresponding nonhomogeneous equation 
\begin{equation}
-(p(x)y'(x))' + [q(x)-zr(x)] y(x) = f(x) \;\text{ a.e.~on }\; (a,b),  \lb{23.2.4.11a}
\end{equation}
is given by the expression 
\begin{align}
y(x) &= C y_1(z,x) + D y_2(z,x)    \no \\
& \quad + \int_{x_0}^x r(x') dx' \, \f{[y_1(z,x) y_2(z,x') 
- y_2(z,x) y_1(z,x')]}{W(y_1(z,\dott), y_2(z,\dott))} f(x'),     \lb{23.2.4.11b}  \\
& \hspace*{5.87cm}  z \in \bbC, \; x, x_0 \in (a,b),    \no 
\end{align}
for some constants $C, D \in \bbC$. In addition, one has
\begin{align}
y^{[1]}(x) &= C y_1^{[1]}(z,x) + D y_2^{[1]}(z,x)  \no \\
& \quad + \int_{x_0}^x r(x') dx' \, \f{[y_1^{[1]}(z,x) y_2(z,x') 
- y_2^{[1]}(z,x) y_1(z,x')]}{W(y_1(z,\dott), y_2(z,\dott))} f(x'),      \lb{23.2.4.11c} \\
& \hspace*{6.25cm}  z \in \bbC, \; x, x_0 \in (a,b).    \no
\end{align}
Equations \eqref{23.2.4.11b} and \eqref{23.2.4.11c} will be used in the proof of the principal result of this paper next:

\begin{theorem} \lb{23.t7.3.11}
Assume Hypothesis \ref{23.h2.1} and that $\tau$ is in the limit circle case at $a$ and $b$ $($i.e., $\tau$ is quasi-regular 
on $(a,b)$$)$. In addition, assume that $T_{min} \geq \lambda_0 I$ for some $\lambda_0 \in \bbR$, and denote by 
$u_a(\lambda_0, \dott)$ and $\hatt u_a(\lambda_0, \dott)$ $($resp., $u_b(\lambda_0, \dott)$ and 
$\hatt u_b(\lambda_0, \dott)$$)$ principal and nonprincipal solutions of $\tau u = \lambda_0 u$ at $a$ 
$($resp., $b$$)$, satisfying
\begin{equation}
W(\hatt u_a(\lambda_0,\dott), u_a(\lambda_0,\dott)) = W(\hatt u_b(\lambda_0,\dott), u_b(\lambda_0,\dott)) = 1.  
\lb{23.7.3.33AB} 
\end{equation}
Introducing $v_j \in \dom(T_{max})$, $j=1,2$, via 
\begin{align}
v_1(x) = \begin{cases} \hatt u_a(\lambda_0,x), & \text{for $x$ near a}, \\
\hatt u_b(\lambda_0,x), & \text{for $x$ near b},  \end{cases}   \quad 
v_2(x) = \begin{cases} u_a(\lambda_0,x), & \text{for $x$ near a}, \\
u_b(\lambda_0,x), & \text{for $x$ near b},  \end{cases}   \lb{23.7.3.33A}
\end{align} 
one obtains for all $g \in \dom(T_{max})$, 
\begin{align}
\begin{split} 
\wti g(a) &= - W(v_2, g)(a) = \wti g_1(a) =  - W(u_a(\lambda_0,\dott), g)(a)    \\
&= \lim_{x \downarrow a} \f{g(x)}{\hatt u_a(\lambda_0,x)},    \\
\wti g(b) &= - W(v_2, g)(b) = \wti g_1(b) =  - W(u_b(\lambda_0,\dott), g)(b)   \\
&= \lim_{x \uparrow b} \f{g(x)}{\hatt u_b(\lambda_0,x)},    
\lb{23.7.3.34A} 
\end{split} \\
\begin{split} 
{\wti g}^{\, \prime}(a) &= W(v_1, g)(a) = \wti g_2(a) = W(\hatt u_a(\lambda_0,\dott), g)(a)   \\
&= \lim_{x \downarrow a} \f{g(x) - \wti g(a) \hatt u_a(\lambda_0,x)}{u_a(\lambda_0,x)},    \\ 
{\wti g}^{\, \prime}(b) &= W(v_1, g)(b) = \wti g_2(b) = W(\hatt u_b(\lambda_0,\dott), g)(b)   \\ 
&= \lim_{x \uparrow b} \f{g(x) - \wti g(b) \hatt u_b(\lambda_0,x)}{u_b(\lambda_0,x)}.    \lb{23.7.3.34B}
\end{split} 
\end{align}
In particular, the limits on the right-hand sides in \eqref{23.7.3.34A}, \eqref{23.7.3.34B} exist. 
\end{theorem}
\begin{proof}
It suffices to discuss the limits $x \downarrow a$ in \eqref{23.7.3.34A}, \eqref{23.7.3.34B} as the limits $x \uparrow b$ are treated in precisely the same manner. Let $g \in \dom(T_{max})$, then the variation of constants formulas \eqref{23.2.4.11b}, \eqref{23.2.4.11c} yield for $c \in (a,b)$, 
\begin{align}
g(x) &= c_{1,a} u_a(\lambda_0,x) + c_{2,a} \hatt u_a(\lambda_0,x)    \no \\
& \quad - u_a(\lambda_0,x) \int_c^x r(x') dx' \, \hatt u_a(\lambda_0,x') ((\tau - \lambda_0)g)(x')   \lb{23.7.3.36A} \\
& \quad + \hatt u_a(\lambda_0,x) \int_c^x r(x') dx' \, u_a(\lambda_0,x') ((\tau - \lambda_0)g)(x'),    \no \\
g^{[1]}(x) &= c_{1,a} u_a^{[1]}(\lambda_0,x) + c_{2,a} \hatt u_a^{[1]}(\lambda_0,x)    \no \\
& \quad - u_a^{[1]}(\lambda_0,x) \int_c^x r(x') dx' \, \hatt u_a(\lambda_0,x') ((\tau - \lambda_0)g)(x')   \lb{23.7.3.37A} \\
& \quad + \hatt u_a^{[1]}(\lambda_0,x) \int_c^x r(x') dx' \, u_a(\lambda_0,x') ((\tau - \lambda_0)g)(x').    \no 
\end{align}
Thus,
\begin{align} 
\begin{split}
W(u_a(\lambda_0, \dott), g)(x) &= - c_{2,a} - \int_c^x r(x') dx' \, u_a(\lambda_0,x') ((\tau - \lambda_0)g)(x'), \lb{23.7.3.38A} \\
W(\hatt u_a(\lambda_0, \dott), g)(x) &= - c_{1,a} + \int_c^x r(x') dx' \, \hatt u_a(\lambda_0,x') ((\tau - \lambda_0)g)(x'), 
\end{split} 
\end{align}
and hence the limits $x \downarrow a$ on the right-hand sides of \eqref{23.7.3.38A} exist by the hypothesis 
$v_j, g \in \dom(T_{max})$, $j=1,2$, and by Lemma \ref{23.l4.2}\,$(ii)$,
\begin{align} 
\begin{split}
W(v_1, g)(a) &= W(u_a(\lambda_0, \dott), g)(a)     \\ 
&= - c_{2,a} - \int_c^a r(x') dx' \, u_a(\lambda_0,x') ((\tau - \lambda_0)g)(x'), \lb{23.7.3.39A} \\
W(v_2, g)(a) &= W(\hatt u_a(\lambda_0, \dott), g)(a)    \\
&= - c_{1,a} + \int_c^a r(x') dx' \, \hatt u_a(\lambda_0,x') ((\tau - \lambda_0)g)(x'). 
\end{split} 
\end{align}
Combining \eqref{23.7.3.36A} and \eqref{23.7.3.39A} using the fact that 
$\lim_{x \downarrow a} u_a(\lambda_0,x)/ \hatt u_a(\lambda_0,x) = 0$ (cf.\ \eqref{23.2.9}) then yields
\begin{align}
\begin{split}
\wti g(a) &= \lim_{x \downarrow a} \f{g(x)}{\hatt u_a(\lambda_0,x)} 
= c_{2,a} + \int_c^a r(x') dx' \, u_a(\lambda_0,x') ((\tau - \lambda_0)g)(x')    \\
&= - W(u_a(\lambda_0,\dott),g)(a) = - W(v_1, g)(a).
\end{split} 
\end{align}
Similarly, employing \eqref{23.7.3.33AB} and L'H{\^ o}pital's rule, 
\begin{align}
{\wti g}^{\, \prime}(a) &= \lim_{x \downarrow a} \f{g(x) - \wti g(a) \hatt u_a(\lambda_0,x)}{u_a(\lambda_0,x)} 
=  \lim_{x \downarrow a}\f{[g(x)/\hatt u_a(\lambda_0,x)] - \wti g(a)}{u_a(\lambda_0,x)/\hatt u_a(\lambda_0,x)} 
= \f{0}{0}     \no \\
&= \lim_{x \downarrow a} 
\f{\big[\hatt u_a(\lambda_0,x) g'(x) - \hatt u_a^{\, \prime}(\lambda_0,x) g(x)\big]\big/ \big[\hatt u_a(\lambda_0,x)^2\big]}
{\big[\hatt u_a(\lambda_0,x) u_a'(\lambda_0,x) -\hatt u_a^{\, \prime}(\lambda_0,x) u_a(\lambda_0,x)\big] \big/ 
\big[\hatt u_a(\lambda_0,x)^2\big]} \cdot \f{p(x)}{p(x)}      \no \\
&= W(\hatt u_a(\lambda_0,x), g)(a) = W(v_2,g)(a), 
\end{align}
completing the proof of \eqref{23.7.3.34A}, \eqref{23.7.3.34B} in the case $x \downarrow a$.
\end{proof} 

\begin{remark} \lb {23.r7.311}
The notion of ``generalized boundary values'' in \eqref{23.4.5.16} and \eqref{23.7.3.34A}, \eqref{23.7.3.34B} corresponds to ``boundary values for $\tau$'' in the sense of \cite[p.~1297, 1304--1307]{DS88}, see also 
\cite[Sect.~3]{Fu73}, \cite[p.~57]{Fu77}. \hfill $\diamond$
\end{remark}

The Friedrichs extension $T_F$ of $T_{min}$ now permits a particularly simple characterization in terms of the generalized boundary values $\wti g(a), \wti g(b)$ as derived by Niessen and Zettl \cite{NZ92}(see also \cite{GP79}, \cite{Ka72}, \cite{Ka78}, \cite{KKZ86}, \cite{MZ00}, \cite{Re51}, \cite{Ro85}, \cite{YSZ15}):

\begin{theorem} \lb{23.t7.3.12}
Assume Hypothesis \ref{23.h2.1} and that $\tau$ is in the limit circle case at $a$ and $b$ $($i.e., $\tau$ is quasi-regular 
on $(a,b)$$)$. In addition, assume that $T_{min} \geq \lambda_0 I$ for some $\lambda_0 \in \bbR$. Then the Friedrichs extension $T_F$ of $T_{min}$ is characterized by
\begin{align}
T_F f = \tau f, \quad f \in \dom(T_F)= \big\{g\in\dom(T_{max})  \, \big| \, \wti g(a) = \wti g(b) = 0\big\}.    \lb{23.7.3.42}
\end{align}
\end{theorem}

\begin{remark} \lb{23.r7.3.13}
$(i)$ Our notation $\wti g (d)$, ${\wti g}^{\, \prime}(d)$ in Theorem \ref{23.t7.3.11}, instead of $\wti g_j (d)$, $j=1,2$, in 
Theorem \ref{23.t4.15}, of course resembles the use of the difference quotient in $(p g')(d)$, $d \in \{a,b\}$, in the context where $\tau$ is regular as explained in \eqref{23.1.1}--\eqref{23.1.12}. \\[1mm] 
$(ii)$ The generalized boundary values 
\begin{align}
\wti g(d) &= \lim_{x \to d} \f{g(x)}{\hatt u_d(\lambda_0,x)},   \lb{23.7.3.43} \\
{\wti g}^{\, \prime}(d) & = \lim_{x \to d} \f{g(x) - \wti g(d) \hatt u_d(\lambda_0,x)}{u_d(\lambda_0,x)},     \lb{23.7.3.44}
\end{align}
at an endpoint $d \in \{a,b\}$ have a longer history. They were originally introduced by Rellich \cite{Re43} in connection with coefficients $p, q, r$ that had a very particular behavior in a neighborhood of the endpoint $d$ of the type
\begin{align}
\begin{split}
p(x) &= (x - d)^{\sigma} \big[p_0 + p_1 (x - d) + p_2 (x - d)^2 + \cdots \big],  \\
q(x) &= (x - d)^{\sigma - 2} \big[q_0 + q_1 (x - d) + q_2 (x - d)^2 + \cdots \big],  \\
r(x) &= (x - d)^{\sigma - 2} \big[r_0 + r_1 (x - d) + r_2 (x - d)^2 + \cdots \big],  \\
\end{split}
\end{align}
with $\sigma, p_0, p_1, \dots , q_0, q_1, \dots , r_0, r_1, \dots  \in \bbR$, $p_0 \neq 0$, $r_k \neq 0$ for some 
$k \in \bbN_0$, $k_{\ell} = 0$ for $0 \leq \ell \leq k-1$, 
etc.\footnote{We employ the notation $\bbN_0= \bbN \cup \{0\}$ throughout this paper.} In 1951, Rellich considerably generalized the hypotheses on $p,q,r$. The case of the Bessel equation was reconsidered in \cite{GP84}, and the case of Schr\"odinger operators on $(0,\infty)$ with potentials $q$ satisfying
\begin{equation}
q(x) = \big(\gamma^2 - (1/4)\big) x^{-2} + \eta x^{-1} + \omega x^{-a} + W(x) \quad \text{for a.e.~$x > 0$},
\end{equation}
with $\gamma \geq 0$, $\eta, \omega \in \bbR$, $a \in (0,2)$, and $W \in L^{\infty}((0,\infty))$ real-valued a.e.,  was systematically treated in \cite{BG85}. Niessen and Zettl \cite{NZ92} thoroughly studied this issue under the general Hypothesis \ref{23.h2.1} in Theorems \ref{23.t4.15} and \ref{23.t7.3.11} and, in particular, derived the expression for 
$\wti g(c)$ in \eqref{23.7.3.43}. In this context we also refer to \cite[Propositions~6.11.1, 6.12.1]{BHS19}, which 
discusses ${\wti g}^{\, \prime}(d)$ in terms of boundary triples and identifies $W(\hatt u_b(\lambda_0,\dott), g)(d)$ and ${\wti g}^{\, \prime}(d)$. The analog of ${\wti g}^{\, \prime}(c)$ in \eqref{23.7.3.44} was not considered in \cite{BHS19} and \cite{NZ92}; it is this new wrinkle we now offer in this context of boundary conditions for self-adjoint singular (quasi-regular) Sturm--Liouville operators bounded from below.\\

The analog of ${\wti g}^{\, \prime}(d)$ in \eqref{23.7.3.44} was not considered in \cite{NZ92} and is the small new wrinkle this paper offers in this context of boundary conditions for self-adjoint singular Sturm--Liouville operators. The final explicit limit relation $\lim_{x \to d} \dots$ on the right-hand sides in \eqref{23.7.3.34B} appears to be a new contribution of this paper. \\[1mm] 
$(iii)$ As in \eqref{23.4.5.24}, one readily verifies for $g, h \in \dom(T_{max})$,
\begin{equation}
\wti g(d) {\wti h}^{\, \prime}(d) - {\wti g}^{\, \prime}(d) \wti h(d) = W(g,h)(d), \quad d \in \{a,b\},    \lb{23.7.3.44A} 
\end{equation} 
again interpreted in the sense that either side in \eqref{23.7.3.44A} has a finite limit as $d \downarrow a$ and 
$d \uparrow b$. \\[1mm] 
$(iv)$ While the generalized boundary values at the endpoint $d \in \{a,b\}$ clearly depend on the choice of nonprincipal solution $\hatt u_{d}(\lambda_0, \dott)$ of $\tau u = \lambda_0 u$ at $d$, the Friedrichs boundary conditions $\wti g(a) = \wti g(b) = 0$ are independent of the choice of this nonprincipal solution. That 
$\wti g(a) = \wti g(b) = 0$ represent the Friedrichs boundary condition was recognized by Rellich \cite{Re51} (he used slightly stronger assumptions on the coefficients $p,q,r$ than those in Hypothesis \ref{23.h2.1}). \\[1mm]
$(v)$ As always in this context, if $\tau$ is in the limit point case at one (or both) interval endpoints, the corresponding boundary conditions at that endpoint are dropped and only a separated boundary condition at the other end point (if the latter is a limit circle endpoint for $\tau$), has to be imposed in Theorems \ref{23.t7.3.11} and \ref{23.t7.3.12}. 
In the case where $\tau$ is in the limit point case at both endpoints, all boundary values and boundary conditions become superfluous as in this case $T_{min} = T_{max}$ is self-adjoint. \hfill $\diamond$
\end{remark}

\section{A Few Remarks on Weyl--Titchmarsh Functions} \lb{23.s5}

In this short section we briefly make contact with Weyl--Titchmarsh functions and demonstrate that the generalized boundary values in \eqref{23.7.3.34A}, \eqref{23.7.3.34B} naturally fit into this framework. 
For simplicity we single out the left endpoint $a$ in the following as the endpoint $b$ can be treated in precisely the same manner.  

\begin{hypothesis} \lb{23.h7.3.11a}
In addition to Hypothesis \ref{23.h2.1}, let $T_{\alpha_0, \beta_0}$, $\alpha_0, \beta_0 \in [0,\pi)$, be any self-adjoint extension of $T_{min}$ in $L^2((a,b); r dx)$ with separated boundary conditions as in \eqref{23.3.52A} 
$($if any\,$)$, and suppose that for some $($and hence for all\,$)$ $c \in (a,b)$, the self-adjoint operator 
$T_{\alpha_0, 0, a,c}$ in $L^2((a,c); r dx)$, associated with $\tau|_{(a,c)}$ and a Dirichlet boundary condition at $c$ 
$($i.e., $g(c)=0$, $g \in \dom(T_{max,a,c})$, the maximal operator associated with $\tau|_{(a,c)}$ in $L^2((a,c); rdx)$$)$, has purely discrete spectrum.
\end{hypothesis}

It has been shown in \cite{EGNT13} and \cite{KST12} that Hypothesis \ref{23.h7.3.11a} is equivalent to the existence of an entire solution $\phi_{\alpha_0}(z, \dott)$ of $\tau u = z u$, $z \in \bbC$, that is real-valued for $z \in \bbR$, and lies in 
$\dom(T_{\alpha_0, \beta_0})$ near the point $a$. In particular, $\phi_{\alpha_0}(z, \dott)$ satisfies the boundary condition indexed by $\alpha_0$ at the left endpoint $a$ if $\tau$ is in the limit circle case at $a$, and 
$\phi_{\alpha_0}(z, \dott) \in L^2((a,c); r dx)$ if $\tau$ is in the limit point case at $a$. In addition, a second, linearly independent entire solution $\theta_{\alpha_0}(z, \dott)$ of $\tau u = z u$ exists, with $\theta_{\alpha_0}(z, \dott)$ real-valued for $z \in \bbR$, satisfying
\begin{equation}
W(\theta_{\alpha_0}(z, \dott),\phi_{\alpha_0}(z, \dott)) =1, \quad z \in \bbC.
\end{equation}

We note that $\phi(z, \dott)$ is unique up to a nonvanishing entire factor (real on the real line) with respect to $z \in \bbC$. 
Hence, we may normalize $\phi_{\alpha_0}(z, \dott)$ such that
\begin{align}
\wti \phi_{\alpha_0} (z,a) &= - \sin(\alpha_0), \quad \wti \phi_{\alpha_0}^{\, \prime} (z,a) = \cos(\alpha_0), \quad 
z \in \bbC,   \lb{23.7.3.44AA}
\intertext{and thus,}
\wti \theta_{\alpha_0} (z,a) &= \cos(\alpha_0), \quad \;\;\,\, \wti \theta_{\alpha_0}^{\, \prime} (z,a) = \sin(\alpha_0), \quad 
z \in \bbC,     \lb{23.7.3.44BB}
\end{align}
normalizations we assume for the rest of this section. 

\begin{remark} \lb{23.r7.3.11b}
As kindly pointed out to us by Charles Fulton \cite{Fu20}, if $\tau$ is in the limit circle case and nonoscillatory at the endpoint $a$ (the most relevant case in this paper), Hypothesis \ref{23.h7.3.11a} is well-known to be satisfied and the existence of $\phi_{\alpha}(z,\dott)$ satisfying \eqref{23.7.3.44AA} and being entire with respect to $z$ has been discussed in \cite[p.~16--17]{Fu73}, \cite{Fu77}. For strongly singular situations implying the limit point case of $\tau$ at $a$ we refer, for instance, to \cite{EGNT13}, \cite{GZ06}, \cite{Ko49}, \cite{KST12}, and \cite{KT11}.  \hfill $\diamond$
\end{remark}

In addition to the entire fundamental system $\phi_{\alpha_0}(z, \dott), \theta_{\alpha_0}(z, \dott)$ of $\tau u = z u$, we also mention the standard entire fundamental system $\theta_0(z, \dott,c), \phi_0(z, \dott,c)$ of $\tau u = z u$ normalized at $c \in (a,b)$ in the usual manner,
\begin{align}
\begin{split} 
& \theta_0(z,c,c) =1, \quad \theta_0^{[1]}(z,c,c) =0,  \\ 
& \phi_0(z,c,c) =0, \quad \phi_0^{[1]}(z,c,c) =1; \quad z \in \bbC,   \lb{23.7.3.45B} 
\end{split}
\end{align}
and the Weyl--Titchmarsh solutions $\psi_{\alpha_0,-} (z,\dott)$ and $\psi_{\alpha_0,+} (z,\dott)$ of $\tau u = z u$ 
given by 
\begin{align}
& \psi_{\alpha_0,-} (z,x) = \theta_0(z,x,c) + m_{\alpha_0,0,-} (z)  \phi_0(z,x,c), \quad z \in \bbC \backslash \sigma(T_{\alpha_0, 0, a,c}), \; x \in (a,b), \\
& \psi_{\beta_0,+} (z,x) = \theta_0(z,x,c) + m_{0,\beta_0,+} (z)  \phi_0(z,x,c), \quad z \in \bbC \backslash \sigma(T_{0,\beta_0, c,b}), \; x \in (a,b), 
\end{align}
where $T_{0,\beta_0, c,b}$ in $L^2((c,b); r dx)$ is the self-adjoint operator associated with $\tau|_{(c,b)}$ and a Dirichlet 
boundary condition at $c$ (i.e., $g(c)=0$, $g \in \dom(T_{max,c,b})$, the maximal operator associated with 
$\tau|_{(c,b)}$ in $L^2((c,b); rdx)$), such that $\psi_{\alpha_0,-} (z, \dott)$ satisfies the boundary condition indexed by 
$\alpha_0$ at the left endpoint $a$ if $\tau$ is in the limit circle case at $a$, and 
$\psi_{\alpha_0,-} (z, \dott) \in L^2((a,c); r dx)$ if $\tau$ is in the limit point case at $a$, and analogously, 
$\psi_{\beta_0,+} (z,\dott)$ satisfies the boundary condition indexed by $\beta_0$ at the right endpoint $b$ if $\tau$ 
is in the limit circle case at $b$, and $\psi_{\beta_0,+} (z,\dott) \in L^2((c,b); r dx)$ if $\tau$ is in the limit point case at $b$.
In particular, $m_{\alpha_0,0,-} (\dott)$ is analytic on $\bbC \backslash \bbR$, meromorphic on $\bbC$, with simple poles on the real axis precisely at the simple eigenvalues of $T_{\alpha_0, 0, a,c}$. Moreover, $\phi_{\alpha_0} (z,\dott)$ is a $z$-dependent multiple of $\psi_{\alpha_0,-} (z,\dott)$, the multiple being entire, real-valued for $z \in \bbR$, and having simple zeros at the simple poles of $m_{\alpha_0,0,-} (\dott)$. In addition, one confirms that
\begin{equation}
m_{\alpha_0,0,-} (z) = \phi_{\alpha_0}^{[1]} (z,c)/\phi_{\alpha_0} (z,c), \quad 
z \in \bbC \backslash \sigma(T_{\alpha_0, 0, a,c}).
\end{equation}

In fact, existence of such an entire fundamental system of solutions of $\tau u = z u$ has been anticipated by Kodaira \cite{Ko49} in 1949, on the basis of an entire fundamental system of solutions defined in \eqref{23.7.3.45B}, but the precise construction was not outlined in \cite{Ko49}. (A rigorous treatment in the case of Schr\"odinger operators with $\phi(z, \dott), \theta(z, \dott)$ analytic in an open neighborhood of the real line was presented in  \cite{GZ06} and in the context of entire functions in \cite{EGNT13}, \cite{KST12}, \cite{KT11}.) 

Next, we will rewrite $\psi_{\beta_0,+}(z,\dott)$ in terms of the entire fundamental system $\phi_{\alpha_0}(z, \dott), \theta_{\alpha_0}(z, \dott)$. Dropping a $z$-dependent (but $x$-independent) factor then finally leads to the singular Weyl--Titchmarsh--Kodaira $m$-function, $m_{\alpha_0,\beta_0}(\dott)$, 
\begin{align} 
\begin{split}
\psi_{\alpha_0,\beta_0} (z,x) &= C_{\alpha_0,\beta_0}(z) \psi_{\beta_0,+} (z,x)   \\
&= \theta_{\alpha_0}(z,x) + m_{\alpha_0,\beta_0} (z) \phi_{\alpha_0}(z,x), \quad 
z \in \bbC \backslash \sigma(T_{\alpha_0,\beta_0,}),     \lb{23.7.3.46A}
\end{split} 
\end{align}
for an appropriate $x$-independent factor $C_{\alpha_0,\beta_0}(\dott)$, which is analytic and nonvanishing on $\bbC \backslash \bbR$.

One can show, as usual, that $m_{\alpha_0,\beta_0} (\dott)$ is analytic on $\bbC \backslash \bbR$ and that 
\begin{equation}
m_{\alpha_0,\beta_0} (z) = \ol{m_{\alpha_0,\beta_0} (\ol{z})}, \quad z \in \bbC_+.
\end{equation}
Moreover, $m_{\alpha_0,\beta_0}$ is a generalized Nevanlinna--Herglotz function (cf.\ \cite{KST12}), and an 
analog of the Stieltjes inversion formula applied to $m_{\alpha_0,\beta_0}$ yields the spectral function 
$\rho_{\alpha_0,\beta_0}$ associated with $T_{\alpha_0, \beta_0}$ (see \cite{EGNT13}, \cite{GZ06}, \cite{KST12}). Even though $m_{\alpha_0,\beta_0}$, and hence, $\rho_{\alpha_0,\beta_0}$, is nonunique, the measure equivalence class generated by the spectral function $\rho_{\alpha_0,\beta_0}$ is unique and hence the spectrum (and its subdivisions) are related to the singularity structure of $m_{\alpha_0,\beta_0}$ on the real line (again, see 
\cite{EGNT13}, \cite{GZ06}, \cite{KST12}). 

In case $\tau$ is in the limit circle case at $a$ and in the limit point case at $b$, one simply drops the $\beta_0$-dependence of all quantities.  

The normalization chosen in \eqref{23.7.3.44AA}, \eqref{23.7.3.44BB} combined with \eqref{23.7.3.46A} readily implies
\begin{equation}
m_{\alpha_0,\beta_0} (z) = 
\wti \psi_{\alpha_0,\beta_0}^{\, \prime} (z,a) \cos(\alpha_0) - \wti \psi_{\alpha_0,\beta_0} (z,a) \sin(\alpha_0), 
\quad z \in \bbC \backslash \sigma(T_{\alpha_0, \beta_0}),
\end{equation}
a result familiar from the special case where $a$ is a regular endpoint (employing the fact \eqref{23.4.5.23}). This illustrates once more that the generalized boundary values \eqref{23.7.3.34A}, \eqref{23.7.3.34B}, in the context of a singular endpoint $a$, are natural extensions of the familiar boundary values in the case of regular endpoints.

Fixing the boundary condition indexed by $\beta_0 \in [0, \pi)$ (if any), and varying the boundary condition at the left endpoint $a$ then yields the standard linear fractional transformation 
\begin{align}
\begin{split} 
m_{\alpha_{1}, \beta_0}(z) =\frac{-\sin(\alpha_{1}-\alpha_{0}) +
\cos(\alpha_{1}-\alpha_{0}) m_{\alpha_{0},\beta_0}(z)}
{\cos(\alpha_{1}-\alpha_{0}) +\sin(\alpha_{1}-\alpha_{0})
m_{\alpha_{0}, \beta_0}(z)},& \\ 
\alpha_1, \alpha_0, \beta_0 \in [0,\pi), \; z \in \bbC \backslash \bbR.&
\end{split}  
\end{align}

We conclude this section with an interesting observation when $\tau$ is in the limit circle case at $a$. In this situation,  Hypothesis \ref{23.h7.3.11a} is always satisfied, and one concludes (see \cite{EGNT13}, \cite{KT11}) that 
\begin{equation}
\f{\Im(m_{\alpha_0,\beta_0} (z))}{\Im(z)} = \int_a^b r(x)dx \, |\psi_{\alpha_0,\beta_0} (z,x)|^2 > 0, \quad 
z \in \bbC \backslash \sigma(T_{\alpha_0, \beta_0}).     \lb{23.5.12} 
\end{equation}
In particular, in this special case $m_{\alpha_0,\beta_0} (\dott)$ is actually a Nevanlinna--Herglotz function (also called a Pick function).

\section{Some Examples} \lb{23.s6}

The following examples illustrate Theorem \ref{23.t7.3.11} in several representative cases, including the Bessel operator on $(0,\infty)$, the Legendre operator on $(-1,1)$, and the Kummer operator on $(0,\infty)$.

As the Bessel operator has been studied in numerous sources, we confine ourselves to a fairly short treatment in this case. The cases of the Legendre and Kummer operators received somewhat less attention in the literature and hence we discuss them in more detail, including the explicit derivation of the underlying $m$-function. 

We start with the Bessel operator in $L^2((0,\infty); dx)$ (see, e.g., \cite[p.~544--547]{AG81}, \cite{AB16}, \cite{AB15}, \cite{Br84}, \cite{BDG11}, \cite{BG85}, \cite{DR18}, \cite[p.~1532--1536]{DS88}, \cite[Sect.~12]{Ev05}, \cite{EK07}, \cite{Fu08}, \cite{FL10}, \cite{GP79}, \cite{GP84}, \cite{GZ06}, \cite[Sect.~7.2]{GTV12}, \cite{KLP06}, \cite{KT11}, \cite{Na74}, \cite{NZ92}, \cite{Pi77}, \cite{Pi79}, \cite{Re43}, \cite{Re51}, \cite[p.~81--90]{Ti62}, \cite[p.~246, 278]{Ze05}; some of these references consider subintervals of 
$(0,\infty)$): 

\subsection{The Bessel Equation on $(0,\infty)$}\lb{23.e7.3.12}
Let $a=0$, $b = \infty$, 
\begin{equation}
p(x) = r(x) = 1, \quad q (x) := q_{\ga}(x) = \f{\gamma^2 - (1/4)}{x^2}, \quad \gamma \in [0, 1), \; x\in(0,\infty).  
\end{equation}
Then $\tau_{\ga} = - d^2/dx^2 +\big[\gamma^2 - (1/4)\big] x^{-2}$, $\ga \in [0,1)$, $x \in (0,\infty)$, 
is in the limit circle case at the endpoint $0$ and in the limit point case at $\infty$. By \eqref{23.4.5.23} it suffices to focus on  the generalized boundary values at the singular endpoint $x = 0$. To this end, we introduce principal and nonprincipal solutions $u_{0,\ga}(0, \dott)$ and $\hatt u_{0,\ga}(0, \dott)$ of $\tau_{\gamma} u = 0$ at $x=0$ by
\begin{align}
u_{0,\ga}(0, x) &= x^{(1/2) + \ga}, \; \ga \in [0,1), \; x \in (0,1),   \\
\hatt u_{0,\ga}(0, x) &= \begin{cases} (2 \ga)^{-1} x^{(1/2) - \ga}, & \ga \in (0,1), \\
x^{1/2} \ln(1/x), & \ga =0;  \end{cases} \quad x \in (0,1).
\end{align}
The generalized boundary values for $g \in \dom(T_{max,\ga})$ $($the maximal operator 
associated with $\tau_{\gamma}$$)$ are then of the form
\begin{align}
\begin{split} 
\wti g(0) &= - W(u_{0,\ga}(0, \dott), g)(0)   \\
&= \begin{cases} \lim_{x \downarrow 0} g(x)\big/\big[(2 \ga)^{-1}x^{(1/2) - \ga}\big], & \ga \in (0,1), \\[1mm]
\lim_{x \downarrow 0} g(x)\big/\big[x^{1/2} \ln(1/x)\big], & \ga =0, 
\end{cases} 
\end{split} \\
\begin{split} 
\wti g^{\, \prime} (0) &= W(\hatt u_{0,\ga}(0, \dott), g)(0)   \\
&= \begin{cases} \lim_{x \downarrow 0} \big[g(x) - \wti g(0) (2 \ga)^{-1} x^{(1/2) - \ga}\big]\big/x^{(1/2) + \ga}, 
& \ga \in (0,1), \\[1mm]
\lim_{x \downarrow 0} \big[g(x) - \wti g(0) x^{1/2} \ln(1/x)\big]\big/x^{1/2}, & \ga =0.
\end{cases}
\end{split} 
\end{align}
Moreover, choosing $\alpha_0 = 0$ for simplicity, one obtains 
\begin{align}
& \phi_0(z,x;\gamma) = \begin{cases} \pi 2^{\gamma} \gamma \Gamma(1-\gamma)^{-1} [\sin(\pi \gamma)]^{-1} z^{- \gamma/2} x^{1/2} J_{\gamma}\big(z^{1/2} x\big), & \gamma \in (0,1), \\[1mm]
x^{1/2} J_0\big(z^{1/2} x\big), & \gamma = 0, 
\end{cases}    \no \\
& \hspace*{8.05cm} z \in \bbC, \; x \in (0,\infty),    \\
& \theta_0(z,x;\gamma) = \begin{cases} 2^{-\gamma - 1} \gamma^{-1} \Gamma(1 - \gamma) 
z^{\gamma/2} x^{1/2} 
J_{-\gamma}\big(z^{1/2} x\big), & \gamma \in (0,1), \\[1mm]
(\pi/2) x^{1/2} \big[- Y_0\big(z^{1/2} x\big) 
+ F(z) J_0\big(z^{1/2} x\big)\big], & \gamma =0, 
\end{cases}      \no \\ 
& \hspace*{7.2cm} z \in \bbC, \; x \in (0,\infty),    \lb{23.7.3.49C} \\
& W(\theta_0(z,\dott;\gamma), \phi_0(z,\dott;\gamma)) =1, \quad z \in \bbC, \\
& \psi_0(z,x;\gamma) = \begin{cases} i 2^{-\gamma -1} \gamma^{-1} \Gamma(1 - \gamma) [\sin(\pi \gamma)]^{-1} 
z^{\gamma/2} x^{1/2} H^{(1)}_{\gamma}\big(z^{1/2} x\big), & \gamma \in (0,1), \\[1mm]
i (\pi/2) x^{1/2} H^{(1)}_0\big(z^{1/2} x\big), & \gamma = 0, 
\end{cases}    \no \\
& \hspace*{6.8cm} z \in \bbC \backslash [0,\infty), \; x \in (0,\infty).  \\
\begin{split} 
& m_0(z;\gamma) = \begin{cases} - e^{-i \pi \gamma} 2^{- 2 \gamma - 1} \gamma^{-1} 
[\Gamma(1 - \gamma)/\Gamma(1+\gamma)] z^{\gamma}, & \gamma \in (0,1), \\[1mm]
i (\pi/2) + \ln(2) - \gamma_{E} - 2^{-1} \ln(z), & \gamma = 0,    
\end{cases}    \lb{23.7.3.49A} \\
& \hspace*{5.65cm} z \in \bbC \backslash [0,\infty), \; x \in (0,\infty).  
\end{split} 
\end{align}

Here $J_{\nu}(\dott), Y_{\nu}(\dott)$ are the standard Bessel functions of order $\nu \in \bbR$, 
$H_{\nu}^{(1)}(\dott)$ is the Hankel function of the first kind and of order $\nu$ (cf.\ \cite[Ch.~9]{AS72}), 
we abbreviated 
\begin{equation}
F(z) = \pi^{-1} \ln(z) - 2 \pi^{-1} \ln(2) + 2 \pi^{-1} \gamma_{E}, 
\end{equation}
in \eqref{23.7.3.49C}, $\Gamma(\dott)$ denotes the Gamma function (cf.\ \cite[Ch.~6]{AS72}), and $\gamma_{E} = 0.57721\dots$ represents Euler's constant $($see, e.g., \cite[Ch.~6]{AS72}$)$.

In particular, the result \eqref{23.7.3.49A} coincides with that obtained in \cite{EK07}. In the limit point case where 
$\gamma \geq 1$, one obtains (cf.\  \cite{Fu08}, \cite{FL10}, \cite{GZ06})
\begin{equation}
 m_0(z;\gamma) = \begin{cases} - C_{\gamma} e^{- i \pi \gamma} (2/\pi) \sin(\pi \gamma) z^{\gamma}, 
 & \gamma \in [1,\infty) \backslash \bbN, \\
 C_0 (2/\pi) z^n [i - (1/\pi) \ln(z)], & \gamma =n, \; n \in \bbN, 
 \end{cases} \quad z \in \bbC \backslash [0,\infty).      \lb{23.7.3.49B}
\end{equation}
One confirms that while \eqref{23.7.3.49A} represents a Nevanlinna--Herglotz function in the limit circle case  
$\gamma \in [0,1)$, the limit point case $\gamma \geq 1$ naturally leads to a generalized Nevanlinna--Herglotz function in \eqref{23.7.3.49B}.

Next, we turn to the Legendre operator in $L^2((-1,1); dx)$ (see, e.g., \cite[p.~535--543]{AG81}, \cite[p.~231--236]{BEL15}, \cite[p.~1520--1526]{DS88}, \cite{EL05}, \cite[Sect.~19]{Ev05}, \cite{Fu82}, \cite{Ka92}, \cite{LZ11}, \cite{Mo81}, \cite{NZ92}, 
\cite[p.~75--81]{Ti62}, \cite[pp.~157, 194, 248, 273--277]{Ze05}; some of these references discuss intervals different from $(-1,1)$). 

\subsection{The Legendre Equation on $(-1,1)$} \lb{23.e7.3.13}
Let $a=-1$, $b = 1$, 
\begin{equation}
p(x) = 1 - x^2, \quad  r(x) = 1, \quad q (x) = 0, \quad x\in(-1,1).  
\end{equation}
Then $\tau_{Leg} = - (d/dx) (1 - x^2) (d/dx)$, $x \in (-1,1)$, is in the limit circle case and singular at both endpoints $\pm 1$. 
Principal and nonprincipal solutions $u_{\pm 1, Leg}(0, \dott)$ and $\hatt u_{\pm 1, Leg}(0, \dott)$ of 
$\tau_{Leg} u = 0$ at $\pm 1$ 
are then given by
\begin{align}
u_{\pm 1, Leg} (0,x) = 1, \quad \hatt u_{\pm 1,Leg} (0,x) = 2^{-1} \ln((1-x)/(1+x)), \quad x \in (-1,1). 
\end{align}
The generalized boundary values for $g \in \dom(T_{max, Leg})$ $($the maximal operator 
associated with $\tau_{Leg}$$)$ are then of the form 
\begin{align}
\begin{split} 
\wti g(\pm 1) &= - W(u_{\pm 1, Leg}(0, \dott), g)(\pm 1)    \lb{23.7.3.50} \\
&= - (p g')(\pm 1) 
= \lim_{x \to \pm 1} g(x)\big/\big[2^{-1} \ln((1-x)/(1+x))\big] ,   
\end{split} \\ 
\begin{split} 
\wti g^{\, \prime} (\pm 1) &= W(\hatt u_{\pm 1, Leg}(0, \dott), g)(\pm 1)  \\
&= \lim_{x \to \pm 1} \big[g(x) - \wti g(\pm 1) 2^{-1} \ln((1-x)/(1+x))\big]. \lb{23.7.3.6.16A}
\end{split} 
\end{align}

One observes the curious fact that by combining \eqref{23.7.3.42} and \eqref{23.7.3.50}, the Friedrichs extension $T_{F,Leg}$ of 
$T_{min, Leg}$ $($the minimal operator associated with $\tau_{Leg}$$)$ then satisfies the boundary conditions
\begin{equation}
(p g')(- 1) = (p g')(1) = 0, 
\end{equation}
which resembles the Neumann $($and not the Dirichlet\,$)$ boundary conditions in the context of a regular 
Sturm--Liouville differential expression on the interval $[-1,1]$. However, since $\tau_{Leg}$ is singular at both 
endpoints $\pm 1$, this represents no conundrum.  

In addition, we note that the spectrum of $T_{F,Leg}$ may be computed explicitly based on \cite[Sect.~9\,$(i)$]{Ev05},
\begin{equation}
\sigma(T_{F,Leg}) = \{n^2+n\}_{n\in \bbN_0}.   \lb{23.7.3.6.18}
\end{equation}

Next, we compute the Weyl--Titchmarsh function corresponding to the Friedrichs extension $T_{F,Leg}$.  We begin by determining the solutions $\phi_0(z,\,\cdot\,)$ and $\theta_0(z,\,\cdot\,)$ of $\tau_{Leg}u=zu$, $z\in \bbC$, corresponding to $\alpha_0=0$ in \eqref{23.7.3.44AA} and \eqref{23.7.3.44BB}.  That is, $\phi_0(z,\,\cdot\,)$ and 
$\theta_0(z,\,\cdot\,)$, $z\in \bbC$, satisfy
\begin{equation}\lb{23.7.3.6.18A}
\tau_{Leg} u = zu
\end{equation}
subject to the conditions
\begin{align}
\wti \phi_0(z,-1)&=0,\quad\quad \wti \phi_0'(z,-1)=1,\lb{23.7.3.6.19A}\\
\wti \theta_0(z,-1)&=1,\quad\ \ \, \, \wti \theta_0'(z,-1)=0.\lb{23.7.3.6.20A}
\end{align}
For fixed $z\in \bbC$, the equation in \eqref{23.7.3.6.18A} is a Legendre equation of the form 
\begin{equation}
\big(1-x^2\big) w''(x) - 2 x w'(x) + \Big[\nu(\nu+1) - \mu^2 \big(1 - x^2\big)^{-2}\Big]w(x) = 0, 
\quad x \in (-1,1), 
\end{equation}
see, \cite[8.1.1]{AS72}, with 
\begin{equation}
\mu=0, \quad \nu=\nu(z):= 2^{-1}\big[-1 + (1+4z)^{1/2} \big],
\end{equation}
and we agree to choose the square root branch such that
\begin{equation}
\nu(z) \in \bbC \backslash \{- \bbN\}. 
\end{equation}

Therefore, linearly independent solutions to \eqref{23.7.3.6.18A} are $P_{\nu(z)}(\,\cdot\,)$ and $Q_{\nu(z)}(\,\cdot\,)$, the Legendre functions of the first and second kind of degree $\nu(z)$, respectively $($cf., e.g., \cite[Ch.~8]{AS72}$)$.  In particular,
\begin{equation}
\begin{split}
\phi_0(z,x) &= c_{\phi,P}(z)P_{\nu(z)}(x) + c_{\phi,Q}(z)Q_{\nu(z)}(x),  \\
\theta_0(z,x) &= c_{\theta,P}(z)P_{\nu(z)}(x) + c_{\theta,Q}(z)Q_{\nu(z)}(x),\quad z\in \bbC, \; x\in (-1,1),  \lb{23.7.3.6.22A}
\end{split}
\end{equation}
for an appropriate set of scalars $c_{\phi,P}(z),c_{\phi,Q}(z),c_{\theta,P}(z),c_{\theta,Q}(z)\in \bbC$.  The representation for $\phi_0(z,\,\cdot\,)$ in \eqref{23.7.3.6.22A} and the initial conditions in \eqref{23.7.3.6.19A} yield the following system of equations for the coefficients $c_{\phi,P}(z)$ and $c_{\phi,Q}(z)$:
\begin{equation}
\begin{cases}
0&= c_{\phi,P}(z)\wti P_{\nu(z)}(-1) + c_{\phi,Q}(z)\wti Q_{\nu(z)}(-1)\\
1&= c_{\phi,P}(z)\wti P_{\nu(z)}'(-1) + c_{\phi,Q}(z)\wti Q_{\nu(z)}'(-1),
\end{cases}
\end{equation}
so that
\begin{align}
c_{\phi,P}(z) &= \frac{-\wti Q_{\nu(z)}(-1)}{\wti P_{\nu(z)}(-1)\wti Q_{\nu(z)}'(-1) - \wti P_{\nu(z)}'(-1)\wti Q_{\nu(z)}(-1)},\\
c_{\phi,Q}(z) &= \frac{\wti P_{\nu(z)}(-1)}{\wti P_{\nu(z)}(-1)\wti Q_{\nu(z)}'(-1) - \wti P_{\nu(z)}'(-1)\wti Q_{\nu(z)}(-1)}.
\end{align}
Analogously, one determines
\begin{align}
c_{\theta,P}(z) &= \frac{\wti Q_{\nu(z)}'(-1)}{\wti P_{\nu(z)}(-1)\wti Q_{\nu(z)}'(-1) - \wti P_{\nu(z)}'(-1)\wti Q_{\nu(z)}(-1)},\\
c_{\theta,Q}(z) &= \frac{-\wti P_{\nu(z)}'(-1)}{\wti P_{\nu(z)}(-1)\wti Q_{\nu(z)}'(-1) - \wti P_{\nu(z)}'(-1)\wti Q_{\nu(z)}(-1)}.
\end{align}

For $z\in \rho(T_{F,Leg})$, the Weyl--Titchmarsh function $m_{0,L}(\,\cdot\,)$ is uniquely determined by the requirement that the function $\psi_{0,Leg}(z,\,\cdot\,)$ defined by
\begin{equation}\lb{23.7.3.6.29A}
\psi_{0,Leg}(z,x) = \theta_0(z,x) + m_{0,Leg}(z)\phi_0(z,x),\quad x\in (-1,1),
\end{equation}
satisfies the condition
\begin{equation}\lb{23.7.3.6.30A}
\wti\psi_{0,Leg}(z,1) = 0.
\end{equation}
In view of \eqref{23.7.3.6.29A}, the condition in \eqref{23.7.3.6.30A} then implies
\begin{equation}\lb{23.7.3.6.31A}
m_{0,Leg}(z) = -\frac{\wti \theta_0(z,1)}{\wti \phi_0(z,1)},\quad z\in \rho(T_{F,Leg}).
\end{equation}
Note that $\wti \phi_0(z,1)\neq 0$ for $z\in \rho(T_{F,Leg})$; otherwise, $z$ would be an eigenvalue of 
$T_{F,Leg}$ with $\phi_0(z,\,\cdot\,)$ a corresponding eigenfunction.  The expansions in \eqref{23.7.3.6.22A} imply
\begin{align}
m_{0,Leg}(z) &= -\frac{c_{\theta,P}(z)\wti P_{\nu(z)}(1) + c_{\theta,Q}(z)\wti Q_{\nu(z)}(1)}{c_{\phi,P}(z)\wti P_{\nu(z)}(1) + c_{\phi,Q}(z)\wti Q_{\nu(z)}(1)}\no\\
&= \frac{\wti Q_{\nu(z)}'(-1)\wti P_{\nu(z)}(1) - \wti P_{\nu(z)}'(-1)\wti Q_{\nu(z)}(1)}{\wti Q_{\nu(z)}(-1)\wti P_{\nu(z)}(1) - \wti P_{\nu(z)}(-1)\wti Q_{\nu(z)}(1)},\quad z\in \rho(T_{F,Leg}).\lb{23.7.3.6.32A}
\end{align}

Applying \eqref{23.7.3.6.16A} and the limiting behavior of $P_{\nu(z)}(x)$ as $x\uparrow 1$ 
$($cf., e.g., \cite[Sect.~3.9.2(4)]{EMOT53}$)$, one computes
\begin{align}
\wti P_{\nu(z)}(1) &= \lim_{x\uparrow 1} \frac{2P_{\nu(z)}(x)}{\ln((1-x)/(1+x))}\no\\
&= \lim_{x\uparrow 1} \frac{2}{\ln((1-x)/(1+x))} = 0,\quad z\in \rho(T_{F,Leg}), \lb{23.7.3.6.33A}
\end{align}
In light of \eqref{23.7.3.6.33A}, the expression for $m_{0,Leg}(z)$ in \eqref{23.7.3.6.32A} simplifies to
\begin{equation}\lb{23.7.3.6.35A}
m_{0,Leg}(z) = \frac{\wti P_{\nu(z)}'(-1)}{\wti P_{\nu(z)}(-1)},\quad z\in \rho(T_{F,Leg}).
\end{equation}

The limiting behavior of $P_{\nu(z)}(x)$ as $x\downarrow -1$ can be obtained from the formula (see, 
\cite[p.~198, eq.~(8.16)]{Te96}),
\begin{align}
P_{\nu}(x) &= \f{1}{\Gamma(- \nu) \Gamma(1 + \nu)} 
\sum_{n \in \bbN_0} \f{(-\nu)_n (1+\nu)_n}{[n!]^2} 2^{-n} (1+x)^n   \no \\
& \quad \times [2 \psi(1+n) - \psi(n - \nu) - \psi(n+1+\nu) - \ln((1+x)/2))],    \lb{23.7.3.6.35} \\ 
& \hspace*{6.4cm} x \in (-1,1), \; \nu \in \bbC,  \no 
\end{align} 
where $(\alpha)_n = \Gamma(\alpha + n)/\Gamma(\alpha)$, $n \in \bbN_0$.  
At first sight \eqref{23.7.3.6.35} appears to have possible singularities as $\nu \to m \in \bbZ$, but closer inspection with the help of properties of the Digamma function, $\psi(\dott) = \Gamma'(\dott)/\Gamma(\dott)$ (cf.\ \cite[Ch.~6]{AS72}), reveals that
\begin{align}
P_{\nu}(x) &= \sum_{k \in \bbN_0} \f{\Gamma(-\nu + k) \Gamma(\nu + k + 1)}{\Gamma(-\nu)^2 \Gamma(\nu+1)^2 [k!]^2}
[2 \psi(k+1) - \ln((1+x)/2)]2^{-k} (1+x)^k   \no \\
& \quad - \sum_{k \in \bbN_0} \f{\Gamma'(-\nu + k) \Gamma(\nu + k + 1)}{\Gamma(-\nu)^2 \Gamma(\nu+1)^2 [k!]^2}
2^{-k} (1+x)^k   \no \\
& \quad - \sum_{k \in \bbN_0} \f{\Gamma'(-\nu + k) \Gamma'(\nu + k + 1)}{\Gamma(-\nu)^2 \Gamma(\nu+1)^2 [k!]^2}
2^{-k} (1+x)^k, \quad x \in (-1,1),   \lb{23.7.3.6.36}
\end{align}
and hence as $\nu \to n \in \bbN_0$, only the 2nd term on the right-hand side of \eqref{23.7.3.6.36} yields a nonzero contribution, and as $\nu \to - n -1$, $n \in \bbN_0$, only the 3rd term on the right-hand side of \eqref{23.7.3.6.36} yields a nonzero contribution. More precisely, for $n \in \bbN$, 
\begin{equation}
\lim_{\nu \to n}P_{\nu}(x) = - \sum_{k=0}^n \bigg[\lim_{\nu \to n} 
\f{\Gamma'(-\nu +k)}{\Gamma(-\nu)^2}\bigg] \f{(n+k)!}{[n!]^2[k!]^2} 2^{-k} (1+x)^k, \quad x \in (-1,1),
\end{equation}
and 
\begin{equation}
\lim_{\nu \to - n -1} P_{\nu}(x) = - \sum_{k=0}^n \bigg[\lim_{\nu \to -n-1} 
\f{\Gamma'(\nu+k+1)}{\Gamma(\nu+1)^2}\bigg] \f{(n+k)!}{[n!]^2 [k!]^2} 2^{-k} (1+x)^k, \quad x \in (-1,1).
\end{equation}
Utilizing the fact that $\Gamma(\dott)$ is meromorphic with first-order poles at $z = - m$, $m \in \bbN_0$, with residue 
$(-1)^m/[m!]$, one obtains
\begin{equation}
\lim_{\nu \to n} 
 \f{\Gamma'(-\nu +k)}{\Gamma(-\nu)^2} = (-1)^{n-k-1} \f{[n!]^2}{(n-k)!} = \lim_{\nu \to -n-1} 
\f{\Gamma'(\nu+k+1)}{\Gamma(\nu+1)^2}, \quad 0 \leq k \leq n, 
\end{equation} 
implying 
\begin{align}
\lim_{\nu \to n}P_{\nu}(x) = \lim_{\nu \to - n -1} P_{\nu}(x) 
& = (-1)^n \sum_{k=0}^n (-1)^k \f{(n+k)!}{(n-k)!} [k!]^{-2} 2^{-k} (1+x)^k  \no \\
& = P_n(x), \quad x \in (-1,1). 
\end{align}
Here the last equality follows from \cite[No.~18.5.7]{OLBC10}, taking $\alpha=\beta = 0$, changing $x$ to $-x$, and utilizing $P_n(-x) = (-1)^n P_n(x)$, $x \in (-1,1)$. 

Formula \eqref{23.7.3.6.35} implies 
\begin{align}
& P_{\nu}(x) \underset{x \downarrow -1}{=} \pi^{-1} \sin(\nu \pi)[\ln((1+x)/2) + 2 \gamma_E 
+ 2 \psi(1+\nu) +  \Oh((1+x)\ln(1+x))]    \no \\
& \hspace*{1.65cm} + \cos(\nu \pi)[1 + \Oh(1+x)], \quad \nu \in \bbC,  \lb{23.7.3.6.35B}
\end{align} 
with $\gamma_E = .57721...$ Euler's constant (cf.\ \cite[Ch.~6]{AS72})\footnote{Incidentally, \eqref{23.7.3.6.35B} corrects a misprint in \cite[Sect.~3.9.2(15), p.~164]{EMOT53} and \cite[Sect.~4.8, p.~197]{MOS66}, where $\gamma_E$ instead of $2 \gamma_E$ appears. In the context of \cite[p.~197]{MOS66} this has been pointed out in \cite[p.~1710]{Sz13}. Moreover, the remainder term $\Oh(1-x)$ in \cite[No.~14.8.3]{OLBC10} must be replaced by $\Oh((1-x)\ln(1-x))$ and at the point in time this manuscript was written, the latter fact also applied to the online version at https://dlmf.nist.gov/14.8.}.  
Thus, 
\begin{align}
& \wti P_{\nu(z)}(-1) = \lim_{x\downarrow -1} \frac{2P_{\nu(z)}(x)}{\ln((1-x)/(1+x))}\no\\
& \quad = \lim_{x\downarrow -1} \frac{2\{ \cos(\nu(z)\pi) + \pi^{-1} \sin(\nu(z)\pi) [2 \gamma_E 
+ 2 \psi(1+\nu(z)) + \ln((1+x)/2)]\}}{\ln((1-x)/(1+x))}\no\\
& \quad = - 2 \pi^{-1} \sin(\nu(z)\pi),  \quad z\in \rho(T_{F,Leg}).\lb{23.7.3.6.36A}
\end{align}
As a consequence of \eqref{23.7.3.6.36A}, one applies \eqref{23.7.3.6.16A} and the limiting behavior of $P_{\nu(z)}(x)$ as $x\downarrow -1$ to compute
\begin{align}
\wti P_{\nu(z)}'(-1) &= \lim_{x\downarrow -1} \big[P_{\nu(z)}(x) - \wti P_{\nu(z)}(-1)2^{-1}\ln((1-x)/(1+x))\big]\no\\
&= \lim_{x\downarrow -1} [P_{\nu(z)}(x) + \pi^{-1} \sin(\nu(z)\pi)\ln((1-x)/(1+x))]     \no\\
&= \lim_{x\downarrow -1} \{\cos(\nu(z)\pi) + \pi^{-1} \sin(\nu(z)\pi) [2\gamma_E + 2\psi(1+\nu(z))  \no\\
&\hspace*{1.2cm} + \ln((1+x)/2)] +\pi^{-1} \sin(\nu(z)\pi) \ln((1-x)/(1+x))\}\no\\
&= \cos(\nu(z)\pi) + 2\pi^{-1} \sin(\nu(z)\pi)[\gamma_E + \psi(1+\nu(z))],\quad z\in \rho(T_{F,Leg}).\lb{23.7.3.6.37A}
\end{align}
The pole structure of $m_{0,Leg}(\dott)$ in \eqref{23.7.3.6.35A} $($cf.\ \eqref{23.7.3.6.36A}, \eqref{23.7.3.6.37A}$)$ independently verifies  
$\sigma(T_{F,Leg})$ in  \eqref{23.7.3.6.18}.  Finally, \eqref{23.7.3.6.35A}, \eqref{23.7.3.6.36A}, and \eqref{23.7.3.6.37A} combine to yield
\begin{align}
m_{0,Leg}(z) &= \frac{\cos(\nu(z)\pi) + 2\pi^{-1} \sin(\nu(z)\pi)[\gamma_E + \psi(1+\nu(z))]}{- 2 \pi^{-1} \sin(\nu(z)\pi)}\no\\
&= - (\pi/2) \cot(\nu(z)\pi) - \gamma_E - \psi(1+\nu(z)),  \quad z\in \rho(T_{F,Leg}).   \lb{23.6.38} 
\end{align}
According to a private communication by Charles Fulton \cite{Fu20}, \eqref{23.6.38} must coincide with $- A/(2B)$, where $A$ and $B$ are defined in \cite[Eqs.~(1.16), (1.17)]{Fu82}. Employing the fact \cite[No.~6.3.7]{AS72}, this indeed is instantly verified. We also note that the analog of \eqref{23.6.38} on the interval $[0,1)$ has indeed been computed in \cite[eq.~(164)]{Mo81}. This was revisited in \cite{Ka92} from an operator theoretic point of view.

Formula \eqref{23.6.38} displays a difference of Nevanlinna--Herglotz functions. To show that it actually represents a Nevanlinna--Herglotz function one can argue as follows: We start by recalling (cf.\ \cite[p.~28]{Do74})
\begin{equation}
- \pi \cot(z \pi) = \sum_{n \in \bbZ} \bigg[\f{1}{n - z} - \f{n \pi^2}{n^2 \pi^2 + 1}\bigg], 
\quad z \in \bbC \backslash \bbZ,
\end{equation}
and (cf.\ \cite[No.~6.3.16]{AS72})
\begin{equation}
\psi(1+z) = - \gamma_E + \sum_{n \in \bbN} \bigg[\f{1}{n} - \f{1}{n+z}\bigg], \quad z \in \bbC \backslash (-\bbN).
\end{equation}
Thus,
\begin{align}
m_{0,Leg}(z) &= - (\pi/2) \cot(\nu(z)\pi) - \gamma_E - \psi(1+\nu(z))    \no \\
&= \f{1}{2} \sum_{n \in \bbZ} \bigg[\f{1}{n - \nu(z)} - \f{n \pi^2}{n^2 \pi^2 + 1}\bigg] 
+ \sum_{n \in \bbN} \bigg[\f{1}{n + \nu(z)} - \f{1}{n}\bigg]   \no \\
&= - \f{1}{2 \nu(z)} + \f{1}{2} \sum_{n \in \bbN} \bigg[\f{1}{n - \nu(z)} - \f{n \pi^2}{n^2 \pi^2 + 1}\bigg]    \no \\
& \quad + \f{1}{2} \sum_{n \in \bbN} \bigg[\f{-1}{n + \nu(z)} + \f{n \pi^2}{n^2 \pi^2 + 1}\bigg]   \no \\
& \quad + \f{1}{2} \sum_{n \in \bbN} \bigg[\f{1}{n + \nu(z)} - \f{1}{n}\bigg] 
+ \f{1}{2} \sum_{n \in \bbN} \bigg[\f{1}{n + \nu(z)} - \f{1}{n}\bigg]    \no \\  
&= - \f{1}{2 \nu(z)} + \f{1}{2} \sum_{n \in \bbN} \bigg[\f{1}{n - \nu(z)} - \f{1}{n}\bigg] 
+ \f{1}{2} \sum_{n \in \bbN} \bigg[\f{1}{n + \nu(z)} - \f{1}{n}\bigg]    \no \\
&= - \f{1}{2 \nu(z)} + \f{1}{2} \sum_{n \in \bbN} \bigg[\f{1}{n - \nu(z)} - \f{1}{n}\bigg] 
+ \f{1}{2} \bigg[\f{1}{1+ \nu(z)} - 1\bigg]     \no \\
&\quad + \f{1}{2} \sum_{n=2}^{\infty} \bigg[\f{1}{n + \nu(z)} - \f{1}{n}\bigg]    \no \\
&= - \f{1}{2} \f{1}{\nu(z) [\nu(z) + 1]} - \f{1}{2} + \f{1}{2} \sum_{n \in \bbN} \bigg[\f{1}{n - \nu(z)} - \f{1}{n}\bigg] 
\no \\
& \quad + \f{1}{2} \sum_{n \in \bbN} \bigg[\f{1}{n +1 + \nu(z)} - \f{1}{n}\bigg] 
+ \f{1}{2} \sum_{n \in \bbN} \bigg[\f{1}{n} - \f{1}{n +1}\bigg]   \no \\
&= - \f{1}{2z} - \f{1}{2} 
+ \f{1}{2} \sum_{n \in \bbN} \f{1}{n (n+ 1)}    \no \\
& \quad + \f{1}{2} \sum_{n \in \bbN} \bigg[\f{1}{n + 2^{-1} - 2^{-1} (1 + 4z)^{1/2}} - \f{1}{n}\bigg]    \no \\
& \quad + \f{1}{2} \sum_{n \in \bbN} \bigg[\f{1}{n + 2^{-1} + 2^{-1} (1 + 4z)^{1/2}} - \f{1}{n}\bigg]    \no \\
&= - \f{1}{2z} + \sum_{n \in \bbN} \bigg[\f{n+ 2^{-1}}{(n+ 2^{-1})^2 - 4^{-1} -z} - \f{1}{n}\bigg]
\no \\
&= - \f{1}{2z} + \sum_{n \in \bbN} \bigg[\f{n + 2^{-1}}{n(n+1) - z} - \f{1}{n}\bigg], 
\quad z \in \bbC \backslash \{n(n+1)\}_{n \in \bbN_0}.
\end{align}
Here we used (cf.\ \cite[No.~0.2441, p.~10]{GR80})
\begin{equation}
\sum_{n \in \bbN} \f{1}{n(n+1)} = 1,
\end{equation}
and 
\begin{equation}
\nu(z)[\nu(z) + 1] = z, \quad z \in \bbC. 
\end{equation}
Once again, one confirms explicitly that the set of poles of $m_{0,Leg}(\dott)$ coincides with the spectrum of 
$T_{F,Leg}$ as recorded in \eqref{23.7.3.6.18}.

As a final example, we turn the case of the Laguerre equation, also known as the Kummer, or confluent hypergeometric equation (see, e.g. \cite[Sects.~10, 27]{Ev05}, \cite[p.~284]{Ze05}).

\subsection{The Laguerre (or Kummer, resp., Confluent Hypergeometric) Equation on $(0,\infty)$} \lb{23.e7.3.14}
Let $a=0$, $b=\infty$,
\begin{align}
\begin{split} 
p(x):= p_{\beta}(x) = x^{\beta}e^{-x},\quad q(x)=0, \quad r(x):=r_{\beta}(x)=x^{\beta-1}e^{-x},& \\
\beta\in (0,2), \; x\in (0,\infty).& 
\end{split}  
\end{align}
The corresponding differential expression is then given by
\begin{equation}
\tau_{\beta,Lag}=-x^{1-\beta}e^x\frac{d}{dx}x^{\beta}e^{-x}\frac{d}{dx},\quad x\in (0,\infty),
\end{equation}
and the underlying Hilbert space is $L^2((0,\infty);x^{\beta-1}e^{-x}\, dx)$. At $x=0$, $\tau_{\beta,Lag}$ is regular for 
$\beta \in (0,1)$ and singular for $\beta \in [1,2)$. 

For $z\in \bbC$, solutions to the Kummer equation
\begin{equation}\lb{23.6.40}
\tau_{\beta,Lag}y=zy
\end{equation}
are given by (cf., e.g., \cite[13.1.12, 13.1.13]{AS72})
\begin{align}
&y_{1,\beta}(z,x) = F(-z,\beta;x),  \quad \beta \in (0,2), \; z \in \bbC, \; x\in (0,\infty),  \\
&y_{2,\beta}(z,x) = \begin{cases} x^{1-\beta}F(1-\beta-z,2-\beta;x), & \beta \in (0,2)\backslash\{1\}, 
\; z \in \bbC, \\[1mm]
\Gamma(-z) U(-z,1;x), & \beta =1, \; z \in \bbC \backslash\{0\}, \\[1mm] 
\displaystyle{- \int_1^x dt \, t^{-1} e^t}, & \beta =1, \; z = 0, 
\end{cases}    \\
& \hspace*{7.05cm} x \in (0, \infty),    \no 
\end{align}
where $F(\dott,\,\cdot\,;\,\cdot\,)$ (also frequently denoted by $\mathstrut_1 F_1(\dott,\,\cdot\,;\,\cdot\,)$, 
or $M(\dott,\,\cdot\,;\,\cdot\,)$), denotes the confluent hypergeometric function and $U(\dott, 1; \dott)$, 
represents an associated logarithmic case (cf., e.g., \cite[Sect.~13.2]{OLBC10}).  

One notes that 
\begin{equation}
y_{1,\beta}(0,x) = 1, \quad x \in (0,\infty).    \lb{23.6.41} 
\end{equation}

Since\footnote{The case $\beta=1$ in \eqref{23.6.42a} has been misrepresented in several sources. For instance, \cite[No.~13.5.9]{AS72} entirely missed the term $-2\gamma_E$ (this has been pointed out in 
\cite[p.~780]{vHK83}) and \cite[p.~288]{MOS66} have the wrong sign of this term, namely, $+2 \gamma_E$, as has been noted in \cite[p.~777]{vHK83}; the asymptotic formula presented in \cite[No.~13.2.19]{OLBC10} is correct.}
\begin{align}\lb{23.6.42}
& y_{1,\beta}(z,x) \underset{x\downarrow 0}{=} 1 - (z/\beta) x +O\big(x^2\big), \quad \beta \in (0,2), \; z \in \bbC, \\
& y_{2,\beta}(z,x) \underset{x\downarrow 0}{=} \begin{cases} x^{1-\beta} 
\big\{[1+[(1-\beta-z)/(2-\beta)]x + O\big(x^2\big)\big\}, & \beta \in (0,2)\backslash\{1\}, \\ 
& \hspace*{1.25cm} z \in \bbC,  \\[1mm] 
- \ln(x) -[\Gamma'(-z)/\Gamma(-z)] - 2 \gamma_E + \Oh(x|\ln(x)|), & \beta = 1, \; z \in \bbC\backslash\{0\}, \\[1mm] 
- \ln(x) \big\{1 + C_0 [- \ln(x)]^{-1} + x + \Oh\big(x^2\big)\big\}, & \beta = 1, \; z = 0, 
\end{cases}     \lb{23.6.42a} 
\end{align}
where we used integration by parts to obtain
\begin{align}
& \int_x^1 dt \, t^{-1} e^t = - \ln(x) e^x + \int_0^1 dt \, t[1 - \ln(t)] e^t - \int_0^x dt \, t [1 - \ln(t)] e^t  \no \\
& \hspace*{1.75cm} = - \ln(x) \big\{1 + C_0 [-\ln(x)]^{-1} + x + \Oh\big(x^2\big)\big\}, \\ 
& \, C_0 = \int_0^1 dt \, t [1 - \ln(t)] e^t, 
\end{align}
the two solutions $y_{j,\beta}(z,\,\cdot\,)$, $j\in \{1,2\}$, are linearly independent.  Alternatively, their linear independence may be deduced from the Wronskian by applying \cite[13.1.20, 13.1.22]{AS72} as follows:
\begin{align}
&W(y_{1,\beta}(z,\,\cdot\,),y_{2,\beta}(z,\,\cdot\,))(x)\no\\
&\quad = x^{\beta}e^{-x}\big[y_{1,\beta}(z,x)y_{2,\beta}'(z,x) - y_{1,\beta}'(z,x)y_{2,\beta}(z,x)\big]\no\\
&\quad = \begin{cases} 1-\beta, & \beta\in (0,2)\backslash\{1\}, \; z\in \bbC, \\
-1, & \beta = 1, \; z \in \bbC,
\end{cases} \quad  x\in (0,\infty).    \lb{23.6.43}
\end{align}
The limiting relations in \eqref{23.6.42} and \eqref{23.6.42a} imply that for $c=c(z)\in (0,1)$ sufficiently small,
\begin{align}
\begin{split} 
& \int_0^c dx \, |y_{1,\beta}(z,x)|^2 x^{\beta-1} e^{-x} = \int_0^c dx \, |1+\Oh(x)|^2 x^{\beta-1} e^{-x} < \infty,  \\
& \hspace*{6.85cm} \beta\in (0,2), \; z \in \bbC, \lb{23.6.49A} 
\end{split} \\
\begin{split} 
& \int_0^c dx \, |y_{2,\beta}(z,x)|^2 x^{\beta-1} e^{-x} = \int_0^c dx \, x^{1-\beta}|1+\Oh(x)|^2 e^{-x} < \infty,  \\ 
& \hspace*{6.15cm} \beta\in (0,2)\backslash\{1\}, \; z \in \bbC,  \lb{23.6.49B} 
\end{split} \\
\begin{split} 
& \int_0^c dx \, |y_{2,1}(z,x)|^2 e^{-x} = \int_0^c dx \, [- \ln(x)]^2 \big|1+O\big([-\ln(x)]^{-1}\big)\big|^2 e^{-x} < \infty,  \\ 
& \hspace*{8.1cm} \beta = 1, \; z \in \bbC\backslash\{0\}.  \lb{23.6.49C} 
\end{split} 
\end{align}
Equations \eqref{23.6.49A}--\eqref{23.6.49C} imply that $\tau_{\beta, Lag}$, $\beta \in (0,2)$, is in the limit circle case at $x=0$.  Moreover, the asymptotic relation 
\begin{equation}
F(-z,b;x) \underset{x \uparrow \infty}{=} \frac{\Gamma(b)}{\Gamma(-z)}e^xx^{-b-z} 
\big[1 + \Oh\big(|x|^{-1}\big)\big], 
\quad z\in \bbC \backslash \bbN_0, \; b \in \bbC \backslash (- \bbN_0),    \lb{23.6.52A}
\end{equation}
(cf., e.g., \cite[13.1.4]{AS72}) implies
\begin{equation}
\int_0^{\infty}|y_{1,\beta}(z,x)|^2x^{\beta-1}e^{-x}\, dx = \infty, \quad \beta \in (0,2), \; z\in \bbC \backslash \bbN_0, 
\end{equation}
so that $\tau_{\beta}$, $\beta \in (0,2)$, is in the limit point case at $x=\infty$.

By \cite[Corollary~XI.6.1]{Ha02}, the fact \eqref{23.6.41} implies that the Kummer equation \eqref{23.6.40} is disconjugate (hence, nonoscillatory) on $(0,\infty)$ for $z=0$, and it follows that
\begin{equation}\lb{23.6.44}
\text{\eqref{23.6.40} is disconjugate on $(0,\infty)$ for all $\beta \in (0,2)$, $z \in (-\infty,0]$.}
\end{equation}
In light of \eqref{23.6.44}, the Kummer equation \eqref{23.6.40} possesses principal and nonprincipal solutions at $x=0$ for each $z \in (-\infty,0]$.  A principal solution of $\tau_{\beta,Lag}u=\lambda u$, $\lambda \leq 0$, at $x=0$ is given by
\begin{equation}\lb{23.6.45}
u_{0,\beta,Lag}(\lambda,\,\cdot\,)=
\begin{cases}
(1-\beta)^{-1}y_{2,\beta}(\lambda,\,\cdot\,),&\beta\in (0,1),\\[1mm] 
-(1-\beta)^{-1}y_{1,\beta}(\lambda,\,\cdot\,),&\beta\in (1,2), \\[1mm] 
y_{1,1}(\lambda,\,\cdot\,),& \beta = 1, 
\end{cases} \quad \lambda \leq 0, 
\end{equation}
and a nonprincipal solution of $\tau_{\beta,Lag}u=\lambda u$ at $x=0$ is given by
\begin{equation}\lb{23.6.46}
\widehat{u}_{0,\beta,Lag}(\lambda,\,\cdot\,)=
\begin{cases}
y_{1,\beta}(\lambda,\,\cdot\,),&\beta\in (0,1), \\[1mm] 
y_{2,\beta}(\lambda,\,\cdot\,),&\beta\in [1,2).  
\end{cases} \quad \lambda \leq 0.
\end{equation}
In fact, for $c=c(\lambda)\in (0,1)$ sufficiently small, relation \eqref{23.6.42} implies
\begin{align}
\int_0^c\frac{dx}{p_{\beta}(x)[y_{1,\beta}(\lambda,x)]^2} = \int_0^c\frac{dx}{x^{\beta}e^{-x}|1+\Oh(x)|^2}, 
\quad \lambda \leq 0, 
\end{align}
which is finite for $\beta\in (0,1)$ and infinite for $\beta\in [1,2)$, while relation \eqref{23.6.42a} implies
\begin{align}
&\int_0^c\frac{dx}{p_{\beta}(x)[y_{2,\beta}(\lambda,x)]^2}    \no \\
& \quad = 
\begin{cases} 
\displaystyle{\int_0^c dx \, x^{\beta-2}e^x |1+\Oh(x)|^{-2}}, & \beta \in (0,2) \backslash \{1\}, \; \lambda \leq 0, \\[3mm] 
\displaystyle{\int_0^c dx \, x^{-1} e^x |-\ln(x) + \Oh(1)|^{-2}}, & \beta = 1, \; \lambda < 0, \\[3mm] 
\displaystyle{\int_0^c dx\, x^{-1} e^x \bigg(\int_x^1 dt \, t^{-1} e^t\bigg)^{-2}}, & \beta = 1, \; \lambda = 0,
\end{cases}  
\end{align}
which is infinite for $\beta\in (0,1)$ and finite for $\beta\in [1,2)$.  In addition, the Wronskian relation in \eqref{23.6.43} implies
\begin{align}
&W(\widehat{u}_{0,\beta,Lag}(\lambda,\,\cdot\,),u_{0,\beta,Lag}(\lambda,\,\cdot\,))(x) = (1-\beta)^{-1}W(y_{1,\beta}(\lambda,\,\cdot\,),y_{2,\beta}(\lambda,\,\cdot\,))(x) = 1,  \lb{23.6.49} \no \\
& \hspace*{6.45cm} \beta \in (0,1), \; \lambda \leq 0, \; x\in (0,\infty),    \\
&W(\widehat{u}_{0,\beta,Lag}(\lambda,\,\cdot\,),u_{0,\beta,Lag}(\lambda,\,\cdot\,))(x) = -(1-\beta)^{-1}W(y_{2,\beta}(\lambda,\,\cdot\,),y_{1,\beta}(\lambda,\,\cdot\,))(x) = 1,    \no \\
&\hspace*{6.5cm} \beta \in (1,2), \; \lambda \leq 0, \; x\in (0,\infty),   \lb{23.6.50} \\
&W(\widehat{u}_{0,1,Lag}(\lambda,\,\cdot\,),u_{0,1,Lag}(\lambda,\,\cdot\,))(x) = 1, \quad 
\beta =1, \; \lambda \leq 0, \; x\in (0,\infty).
\end{align}
Thus, 
\begin{align}
& \lim_{x\downarrow 0} \frac{u_{0,\beta}(\lambda,x)}{\widehat{u}_{0,\beta,Lag}(\lambda,x)}= \lim_{x\downarrow 0}\frac{(1-\beta)^{-1}x^{1-\beta}[1+\Oh(x)]}{1+\Oh(x)}=0,\quad \beta\in (0,1), \; \lambda \leq 0,  \lb{23.6.51} \\
& \lim_{x\downarrow 0} \frac{u_{0,\beta,Lag}(\lambda,x)}{\widehat{u}_{0,\beta,Lag}(\lambda,x)} = \lim_{x\downarrow 0}\frac{-(1-\beta)^{-1}[1+\Oh(x)]}{x^{1-\beta}[1+\Oh(x)]}=0,\quad \beta\in (1,2), \; \lambda \leq 0,   \lb{23.6.52} \\
& \lim_{x\downarrow 0} \frac{u_{0,1,Lag}(\lambda,x)}{\widehat{u}_{0,1,Lag}(\lambda,x)} = 
\lim_{x\downarrow 0} \f{1 + \Oh(x)}{[-\ln(x)] + \Oh(1)}= 0, \quad \beta =1, \; \lambda < 0,    \lb{23.6.53} \\
& \lim_{x\downarrow 0} \frac{u_{0,1,Lag}(0,x)}{\widehat{u}_{0,1,Lag}(0,x)} = 
\lim_{x\downarrow 0} \bigg(\int_x^1 dt \, t^{-1} e^t\bigg)^{-1} = 0, \quad \beta =1, \; \lambda = 0.    \lb{23.6.54} 
\end{align}
The identities \eqref{23.6.49}--\eqref{23.6.54} confirm that the (principal/nonprincipal) solutions $u_{0,\beta,Lag}(\lambda,\,\cdot\,)$ and $\widehat{u}_{0,\beta,Lag}(\lambda,\,\cdot\,)$, $\lambda\in (-\infty,0]$, satisfy \eqref{23.2.9}--\eqref{23.2.11}.

The generalized boundary values for $g \in \dom(T_{max, \beta,Lag})$ $($the maximal operator 
associated with $\tau_{\beta,Lag}$$)$ are then of the form 
\begin{align}
\wti g(0) &= - W(u_{0,\beta,Lag}(0,\dott), g)(0) 
= \displaystyle{\lim_{x\downarrow 0}} \frac{g(x)}{\hatt u_{0,\beta,Lag}(0,x)}   
\no \\[1mm]
&=
\begin{cases}
\displaystyle{\lim_{x\downarrow 0}} \frac{g(x)}{y_{1,\beta}(0,x)},\quad \beta\in (0,1),  \\[4mm]
\displaystyle{\lim_{x\downarrow 0}} \frac{g(x)}{y_{2,\beta}(0,x)},\quad \beta\in [1,2),  
\end{cases}    \no \\[2mm]
&=
\begin{cases}
g(0),&\beta\in (0,1),\\[1mm]
\displaystyle{\lim_{x\downarrow 0}}\dfrac{g(x)}{x^{1-\beta}},&\beta\in (1,2), \\[2mm] 
\displaystyle{\lim_{x\downarrow 0}} \frac{g(x)}{[- \ln(x)]}, & \beta =1.
\end{cases} \\[2mm]
\wti g^{\, \prime} (0) &= W(\hatt u_{0,\beta,Lag}(0, \dott), g)(0) 
= \displaystyle{\lim_{x\downarrow 0}\frac{g(x)-\wti g(0)\hatt u_{0,\beta,Lag}(0,x)}{u_{0,\beta,Lag}(0,x)}} \no \\[1mm]
&=
\begin{cases}
\displaystyle{\lim_{x\downarrow 0} \frac{g(x)-\wti g(0)y_{1,\beta}(0,x)}{(1-\beta)^{-1}y_{2,\beta}(0,x)}}, 
&\beta\in (0,1),\\[4mm]
\displaystyle{\lim_{x\downarrow 0} \frac{g(x)-\wti g(0)y_{2,\beta}(0,x)}{-(1-\beta)^{-1}y_{1,\beta}(0,x)}}, 
& \beta\in (1,2), \\[4mm] 
\displaystyle{\lim_{x\downarrow 0} \frac{g(x)-\wti g(0)y_{2,1}(0,x)}{y_{1,1}(0,x)}}, & \beta =1, 
\end{cases}   \no \\[2mm]
&=
\begin{cases}
\displaystyle{\lim_{x\downarrow 0}} \dfrac{g(x)-g(0)}{(1-\beta)^{-1}x^{1-\beta}} 
= \dfrac{0}{0} = \displaystyle{\lim_{x\downarrow 0}} \dfrac{g'(x)}{x^{-\beta}} 
= g^{[1]}(0),&\beta\in(0,1),\\[4mm]
\displaystyle{(\beta - 1) \lim_{x\downarrow 0}} \big[g(x)-\wti g(0)x^{1-\beta}\big], &\beta\in (1,2),  \\[2mm] 
\lim_{x\downarrow 0} \big\{g(x) - \wti g(0) [- \ln(x)]\big\}, &\beta = 1. 
\end{cases}    \lb{23.6.67A} 
\end{align}
\vspace{10pt} 

Next, we turn to the computation of the Weyl--Titchmarsh $m$-function corresponding to the Friedrichs extension $T_{F,\beta,Lag}$ of the minimal operator $T_{min,\beta,Lag}$ generated by $\tau_{\beta,Lag}$ with the boundary condition
\begin{equation}
\dom(T_{F,\beta,Lag}) = \big\{g\in \dom(T_{max,\beta,Lag})\,\big|\, \wti g(0)=0\big\},\quad \beta\in (0,2).
\end{equation}

We begin by determining the solutions $\phi_{0,\beta}(z,\,\cdot\,)$ and $\theta_{0,\beta}(z,\,\cdot\,)$ of 
$\tau_{\beta,Lag}u=zu$, $\beta\in (0,2)$, $z\in \bbC$, corresponding to $\alpha_0=0$ in \eqref{23.7.3.44AA} and \eqref{23.7.3.44BB}.  That is, $\phi_{0,\beta}(z,\,\cdot\,)$ and $\theta_{0,\beta}(z,\,\cdot\,)$, $z\in \bbC$, satisfy $\tau_{\beta} u = zu$, 
subject to the conditions
\begin{align}
\wti \phi_{0,\beta}(z,0)&=0,\quad\quad \wti \phi_{0,\beta}'(z,0)=1, \quad \beta \in (0,2), \; z\in \bbC,  \lb{23.7.3.6.19AA}\\
\wti \theta_{0,\beta}(z,0)&=1,\quad\ \ \, \, \wti \theta_{0,\beta}'(z,0)=0, \quad \beta \in (0,2), \; z \in \bbC.    \lb{23.7.3.6.20AA}
\end{align}
Writing, for each $\beta\in (0,2)$,
\begin{align}\lb{23.6.64}
\begin{split}
\phi_{0,\beta}(z,x) &= c_{\phi,1,\beta}(z)y_{1,\beta}(z,x) + c_{\phi,2,\beta}(z)y_{2,\beta}(z,x),\\
\theta_{0,\beta}(z,x) &= c_{\theta,1,\beta}(z)y_{1,\beta}(z,x) + c_{\theta,2,\beta}(z)y_{2,\beta}(z,x),   
\end{split} \\
& \hspace*{3.5cm} z\in \bbC, \; x\in (0,\infty),   \no
\end{align}
one infers that (for $\beta\in (0,2)$, $z\in \bbC$)
\begin{align}
c_{\phi,1,\beta}(z) &= \frac{-\wti y_{2,\beta}(z,0)}{\wti y_{1,\beta}(z,0)\wti y_{2,\beta}'(z,0) - \wti y_{1,\beta}'(z,0)\wti y_{2,\beta}(z,0)},\lb{23.6.65}\\
c_{\phi,2,\beta}(z) &= \frac{\wti y_{1,\beta}(z,0)}{\wti y_{1,\beta}(z,0)\wti y_{2,\beta}'(z,0) - \wti y_{1,\beta}'(z,0)\wti y_{2,\beta}(z,0)},   \no
\end{align}
and
\begin{align}
c_{\theta,1,\beta}(z) &= \frac{\wti y_{2,\beta}'(z,0)}{\wti y_{1,\beta}(z,0)\wti y_{2,\beta}'(z,0) - \wti y_{1,\beta}'(z,0)\wti y_{2,\beta}(z,0)},\lb{23.6.66}\\
c_{\theta,2,\beta}(z) &= \frac{-\wti y_{1,\beta}'(z,0)}{\wti y_{1,\beta}(z,0)\wti y_{2,\beta}'(z,0) - \wti y_{1,\beta}'(z,0)\wti y_{2,\beta}(z,0)}.  \no
\end{align}
For $\beta\in (0,2)$, the Weyl--Titchmarsh $m$-function $m_{0,\beta,Lag}(\,\cdot\,)$ is uniquely determined by the requirement that the function
\begin{align}
\psi_{0,\beta,Lag}(z,x) := \theta_{0,\beta}(z,x) + m_{0,\beta,Lag}(z)\phi_{0,\beta}(z,x),&\lb{23.6.67}\\
\beta\in (0,2), \; z\in \rho(T_{F,\beta,Lag}), \; x\in (0,\infty),&\no
\end{align}
satisfies 
\begin{equation}\lb{23.6.68}
\psi_{0,\beta,Lag}(z,\,\cdot\,)\in L^2\big((0,\infty); x^{\beta-1}e^{-x}\, dx\big), 
\quad \beta\in (0,2), \; z\in \rho(T_{F,\beta,Lag}).
\end{equation}
We distinguish three cases $\beta\in (0,1)$, $\beta\in (1,2)$, and $\beta = 1$.

Starting with $\beta\in (0,1)$, one computes
\begin{align}
\wti y_{1,\beta}(z,0) &= \lim_{x\downarrow 0} y_{1,\beta}(z,x) = \lim_{x\downarrow 0}\big[1+\Oh(x)\big] = 1, \lb{23.6.69} \\
\wti y_{2,\beta}(z,0) &= \lim_{x\downarrow 0}y_{2,\beta}(z,x) = \lim_{x\downarrow 0} x^{1-\beta}\big[1+\Oh(x)\big] = 0,\quad \beta\in (0,1), \; z\in \bbC,   \lb{23.6.70}
\end{align}
while
\begin{align}
\wti y_{1,\beta}'(z,0) &= \lim_{x\downarrow 0} \frac{y_{1,\beta}(z,x) - \wti y_{1,\beta}(z,0)y_{1,\beta}(0,x)}{(1-\beta)^{-1}y_{2,\beta}(0,x)}   \no \\
&= \lim_{x\downarrow 0} \frac{y_{1,\beta}(z,x) - 1}{(1-\beta)^{-1}y_{2,\beta}(0,x)}\no\\
&= \lim_{x\downarrow 0} \frac{\Oh(x)}{(1-\beta)^{-1}x^{1-\beta}\big[1+\Oh(x)\big]}\no\\
&= \lim_{x\downarrow 0}\frac{\Oh(x^{\beta})}{(1-\beta)^{-1}\big[1+\Oh(x)\big]}\no\\
&= 0,\quad \beta\in (0,1), \; z\in \bbC,    \lb{23.6.71}
\end{align}
and
\begin{align}
\wti y_{2,\beta}'(z,0) &= \lim_{x\downarrow 0}\frac{y_{2,\beta}(z,x) - \wti y_{2,\beta}(z,0)y_{1,\beta}(0,x)}{(1-\beta)^{-1}y_{2,\beta}(0,x)}    \no \\
&= \lim_{x\downarrow 0} \frac{y_{2,\beta}(z,x)}{(1-\beta)^{-1}y_{2,\beta}(0,x)}\no\\
&= \lim_{x\downarrow 0} \frac{\big[1+\Oh(x)\big]}{(1-\beta)^{-1}\big[1+\Oh(x)\big]}\no\\
&= 1-\beta,\quad \beta\in (0,1), \; z\in \bbC.    \lb{23.6.72}
\end{align}
By \eqref{23.6.64}--\eqref{23.6.66} and \eqref{23.6.69}--\eqref{23.6.72}, one obtains
\begin{align}
\phi_{0,\beta}(z,x) &= (1-\beta)^{-1}y_{2,\beta}(z,x),\\
\theta_{0,\beta}(z,x) &= y_{1,\beta}(z,x),\quad \beta\in (0,1), \; z\in \bbC, \; x\in (0,\infty).
\end{align}
Therefore, the requirement in \eqref{23.6.68} can be recast as
\begin{align}
\big[y_{1,\beta}(z,\,\cdot\,) + m_{0,\beta,Lag}(z)(1-\beta)^{-1}y_{2,\beta}(z,\,\cdot\,)\big]\in L^2\big((0,\infty);x^{\beta-1}e^{-x}\, dx\big),&\\
\beta\in (0,1), \; z\in \rho(T_{F,\beta,Lag}),&\no
\end{align}
or as,  
\begin{align}
\int^{\infty}\big| y_{1,\beta}(z,x) + m_{0,\beta,Lag}(z)(1-\beta)^{-1}y_{2,\beta}(z,x)\big|^2 x^{\beta-1}e^{-x}\, dx<\infty,&\lb{23.6.76}\\
\beta\in (0,1), \; z\in \rho(T_{F,\beta,Lag}).&\no
\end{align}
The asymptotic relation \eqref{23.6.52A} then implies for each fixed $\beta\in (0,1)$ and $z\in \rho(T_{F,\beta,Lag})$, 
that the integrand
\begin{equation}
\big| y_{1,\beta}(z,x) + m_{0,\beta,Lag}(z)(1-\beta)^{-1}y_{2,\beta}(z,x)\big|^2 x^{\beta-1}e^{-x}
\end{equation}
in \eqref{23.6.76} behaves at $\infty$ like
\begin{equation}
\bigg|\frac{\Gamma(\beta)}{\Gamma(-z)} + m_{0,\beta,Lag}(z) (1-\beta)^{-1}\frac{\Gamma(2-\beta)}{\Gamma(1-\beta-z)}\bigg|^2x^{-(2\Re(z)+\beta+1)}e^x.   \lb{23.6.79} 
\end{equation}
The expression \eqref{23.6.79} is integrable near $\infty$ with respect to Lebesgue measure $dx$ if and only if
\begin{equation}
\frac{\Gamma(\beta)}{\Gamma(-z)} + m_{0,\beta,Lag}(z) (1-\beta)^{-1}\frac{\Gamma(2-\beta)}{\Gamma(1-\beta-z)} = 0,
\end{equation}
so that
\begin{align}\lb{23.6.81}
\begin{split}
m_{0,\beta,Lag}(z) = \frac{(\beta-1)\Gamma(\beta)\Gamma(1-\beta-z)}{\Gamma(2-\beta)\Gamma(-z)} 
= - \frac{\Gamma(\beta)\Gamma(1-\beta-z)}{\Gamma(1-\beta)\Gamma(-z)},& \\ 
 \beta\in (0,1), \; z\in \rho(T_{F,\beta,Lag}),& 
\end{split} 
\end{align} 
in agreement with \cite[eq.~(6)]{De98} and the choice of $\Gamma_0, \Gamma_1$ therein. 

To verify that $m_{0,\beta,Lag}(\dott)$ is a Nevanlinna--Herglotz function one can argue as follows following 
\cite[p.~1, 5]{EMOT53}. Starting from the celebrated formula
\begin{equation}
\Gamma(z)^{-1} = z e^{\gamma_E z} \prod_{n \in \bbN} [1 + (z/n)] e^{-z/n}, \quad z \in \bbC,
\end{equation}
one derives
\begin{align}
\begin{split}
\f{\Gamma(\zeta_1)}{\Gamma(\zeta_1 + \zeta_2)} = [1 + (\zeta_2/\zeta_1)] e^{\gamma_E \zeta_2} 
\prod_{n \in \bbN} \bigg[1 + \f{\zeta_2}{n + \zeta_1}\bigg] e^{- \zeta_2/n},& \\
\zeta_1 \in \bbC \backslash \{- \bbN_0\}, \; \zeta_2 \in \bbC,& 
\end{split} 
\end{align}
and hence,
\begin{align}
\begin{split}
\f{\Gamma(z_1) \Gamma(z_2)}{\Gamma(z_1 + z_3) \Gamma(z_2 - z_3)} = 
\prod_{n \in \bbN_0} \bigg[1 + \f{z_3}{n + z_1}\bigg] \bigg[1 - \f{z_3}{n + z_2}\bigg],& \\
z_1 \in \bbC \backslash \{- \bbN_0\}, \; z_2 \in \bbC \backslash \{- \bbN_0\}, \; z_3 \in \bbC,  \lb{23.6.97}
\end{split} 
\end{align}
a formula also to be found in \cite[No.~8.3251]{GR80}. Choosing
\begin{equation}
z_1 = \beta, \; z_2 = 1 - \beta - z, \; z_3 = 1 - \beta, 
\end{equation}
this implies
\begin{align}
m_{0,\beta,Lag}(z) &= \f{\beta - 1}{\beta \Gamma(2-\beta)} \f{z}{z-1+\beta} 
\prod_{n \in \bbN} \f{n(n+1)}{(n+\beta)(n+1-\beta)}    \no \\
& \quad \times \prod_{n \in \bbN} \f{n+1-\beta}{n} \bigg[1 - \f{1-\beta}{n+1-\beta-z}\bigg]  \no \\
&= C_1(\beta) \f{\beta - 1}{\beta \Gamma(2-\beta)} \f{z}{z-(1-\beta)} 
\prod_{n \in \bbN} \bigg[1 - \f{z}{n}\bigg] \bigg[1 - \f{z}{n+1-\beta}\bigg]^{-1},   \no \\
& \hspace*{4.8cm} z \in \bbC \backslash \{- \beta + \bbN\}, \; \beta \in (0,1),    \lb{23.6.99}
\end{align}
where we abbreviated
\begin{equation}
C_1(\beta) = \prod_{n \in \bbN} \f{n(n+1)}{(n+\beta)(n+1-\beta)} > 0, \quad \beta \in (0,1).
\end{equation}
One verifies that all zeros 
\begin{equation} 
a_n = n, \quad n \in \bbN_0,  
\end{equation} 
and poles 
\begin{equation} 
b_n = n+1 - \beta, \quad n \in \bbN_0, 
\end{equation} 
of $m_{0,\beta,Lag}(\, \cdot \,)$ are simple, the zeros and poles interlace (as $\beta \in (0,1)$), the residues 
at all the poles are strictly negative, and $m_{0,\beta,Lag}(\, \cdot \,)$ is real-valued on $\bbR$.
This corresponds to the situation discussed in \cite[Theorem~1 on p.~308]{Le80} (a result attributed to an unpublished paper by M.~G.~Krein), identifying 
\begin{align}
c < 0, \quad a_{-n}=b_{-n} = - \infty, \; n \in \bbN, \quad a_n = n, \; b_n = n+1-\beta, \; n \in \bbN_0,  
\end{align}
such that $a_n < b_n < a_{n+1}$, $n \in \bbN_0$. At any rate, these properties demonstrate the 
Nevanlinna--Herglotz property of $m_{0,\beta,Lag}(\, \cdot \,)$ for $\beta \in (0,1)$. 
This completes the treatment for the case $\beta\in (0,1)$.

Next, we consider the case $\beta\in (1,2)$.  In analogy to \eqref{23.6.69}--\eqref{23.6.72}, one computes
\begin{align}\lb{23.6.83}
\begin{split}
\wti y_{1,\beta}(z,0) &= 0,\quad\quad\quad \wti y_{2,\beta}(z,0) = 1,\\
\wti y_{1,\beta}'(z,0) &= \beta-1,\quad \, \wti y_{2,\beta}'(z,0) = 0,\quad \beta \in (1,2), \; z\in \bbC.
\end{split}
\end{align}
Thus, \eqref{23.6.64}--\eqref{23.6.66} and \eqref{23.6.83} imply
\begin{align}
\phi_{0,\beta}(z,x) &= (\beta-1)^{-1} y_{1,\beta}(z,x),\\
\theta_{0,\beta}(z,x) &= y_{2,\beta}(z,x),\quad \beta\in (1,2), \; z\in \bbC, \; x\in (0,\infty).
\end{align}
The requirement in \eqref{23.6.68} may be recast as
\begin{align}
\big[y_{2,\beta}(z,\,\cdot\,) + m_{0,\beta,Lag}(z)(\beta-1)^{-1}y_{1,\beta}(z,\,\cdot\,)\big]\in L^2\big((0,\infty);x^{\beta-1}e^{-x}\, dx\big),&\\
\beta\in (1,2), \; z\in \rho(T_{F,\beta,Lag}),&\no
\end{align}
or as, 
\begin{align}
\int^{\infty}\big| y_{2,\beta}(z,x) + m_{0,\beta,Lag}(z)(\beta-1)^{-1}y_{1,\beta}(z,x)\big|^2 x^{\beta-1}e^{-x}\, dx<\infty,&\lb{23.6.87}\\
\beta\in (1,2), \; z\in \rho(T_{F,\beta,Lag}).&\no
\end{align}
The asymptotic relation \eqref{23.6.52A} then implies for each $\beta\in (1,2)$ and $z\in \rho(T_{F,\beta,Lag})$, 
that the integrand
\begin{align}
\big| y_{2,\beta}(z,x) + m_{0,\beta,Lag}(z)(\beta-1)^{-1}y_{1,\beta}(z,x)\big|^2 x^{\beta-1}e^{-x}
\end{align}
in \eqref{23.6.87} behaves at $\infty$ like
\begin{align}
\bigg|\frac{\Gamma(2-\beta)}{\Gamma(1-\beta-z)}+m_{0,\beta,Lag}(z)(\beta-1)^{-1}\frac{\Gamma(\beta)}{\Gamma(-z)}\bigg|^2x^{-(2\Re(z)+\beta+1)}e^x.    \lb{23.6.88}
\end{align}
The expression \eqref{23.6.88} is integrable near $\infty$ with respect to Lebesgue measure $dx$ if and only if 
\begin{equation}
\frac{\Gamma(2-\beta)}{\Gamma(1-\beta-z)}+m_{0,\beta,Lag}(z)(\beta-1)^{-1}\frac{\Gamma(\beta)}{\Gamma(-z)}=0,
\end{equation}
so that
\begin{align}\lb{23.6.91}
\begin{split} 
m_{0,\beta,Lag}(z) = \f{(1-\beta)\Gamma(2-\beta)\Gamma(-z)}{\Gamma(\beta)\Gamma(1-\beta-z)} 
= - \f{\Gamma(2-\beta) \Gamma(-z)}{\Gamma(\beta-1) \Gamma(1-\beta-z)},&   \\
\beta\in (1,2), \; z\in \rho(T_{F,\beta,Lag}),&
\end{split} 
\end{align} 
in agreement with \cite[eq.~(6)]{De98} and the choice of $\Gamma_0, \Gamma_1$ therein. 

To prove that $m_{0,\beta,Lag}(\dott)$ is a Nevanlinna--Herglotz function in accordance with \eqref{23.5.12} one 
can follow the case $\beta \in (0,1)$ step by step: Choosing 
\begin{equation}
z_1 = 1, \; z_2 = -z, \; z_3 = \beta -1, 
\end{equation}
in \eqref{23.6.97}, 
one obtains in complete analogy to \eqref{23.6.99}, 
\begin{align}
m_{0,\beta,Lag}(z) &= \beta (1 - \beta) \Gamma(2 - \beta) \f{z + \beta -1}{z} 
\prod_{n \in \bbN} \f{(n+\beta)(n+1- \beta)}{n(n+1)}    \no \\
& \quad \times \prod_{n \in \bbN} \f{n}{n+1-\beta} \bigg[1 - \f{\beta - 1}{n-z}\bigg]  \no \\
&= C_2(\beta) \beta (1 - \beta) \Gamma(2-\beta) \f{z - (1-\beta)}{z} 
\prod_{n \in \bbN} \bigg[1 - \f{z}{n+1-\beta}\bigg] \bigg[1 - \f{z}{n}\bigg] ^{-1},   \no \\
& \hspace*{5.5cm} z \in \bbC \backslash \bbN_0, \; \beta \in (1,2),    \lb{6.3.114}
\end{align}
where we abbreviated
\begin{equation}
C_2(\beta) = \prod_{n \in \bbN} \f{(n+\beta)(n+1- \beta)}{n(n+1)} > 0, \quad \beta \in (1,2).
\end{equation}
One verifies that all zeros 
\begin{equation} 
a_n = n+1-\beta, \quad n \in \bbN_0,  
\end{equation} 
and poles 
\begin{equation} 
b_n = n, \quad n \in \bbN_0, 
\end{equation} 
of $m_{0,\beta,Lag}(\, \cdot \,)$ are simple, the zeros and poles interlace (as $\beta \in (1,2)$), the residues 
at all the poles are strictly negative, and $m_{0,\beta,Lag}(\, \cdot \,)$ is real-valued on $\bbR$.
This corresponds to the situation discussed in \cite[Theorem~1 on p.~308]{Le80}, identifying 
\begin{align}
c < 0, \quad a_{-n}=b_{-n} = - \infty, \; n \in \bbN, \quad a_n = n+1-\beta, \; b_n = n, \; n \in \bbN_0,  
\end{align}
such that $a_n < b_n < a_{n+1}$, $n \in \bbN_0$. Again, these properties demonstrate the 
Nevanlinna--Herglotz property of $m_{0,\beta,Lag}(\, \cdot \,)$ for $\beta \in (1,2)$. 
This completes the discussion of the case $\beta\in (1,2)$.

\begin{remark} \lb{23.r6.1}
In our original discussion of the Nevanlinna--Herglotz property of $m_{0,\beta,Lag}(\dott)$ for $\beta \in (0,1)$ (cf.\ 
\eqref{23.6.81}) and $\beta \in (1,2)$ (cf.\ \eqref{23.6.91}), we relied on reference \cite[p.~5]{MOS66}. However, as was kindly pointed out to us by Christian Berg \cite{Be19}, the monograph \cite{MOS66} left open the restrictions on the parameter $a$ in their 3rd and enlarged edition. (The condition $a > 0$ apparently appears in the earlier 1946 edition.) Hence we chose an alternative route of proof and based our arguments on formula \eqref{23.6.97}. A systematic approach to ratios of Gamma functions, their relations to Stieltjes and generalized Stieltjes functions and to completely monotonic functions, can be found in \cite{Be08} and \cite{BKP19}.   \hfill $\diamond$
\end{remark}

Finally, we treat the case $\beta = 1$. One computes
\begin{align}
\wti y_{1,1}(z,0) &= \lim_{x\downarrow 0}\frac{y_{1,1}(z,x)}{[-\ln(x)]} = \lim_{x\downarrow0}\frac{1-zx+\Oh(x^2)}{[-\ln(x)]}=0,\quad z\in \bbC,\lb{23.6.98y}\\
\wti y_{2,1}(z,0) &= \lim_{x\downarrow0}\frac{y_{2,1}(z,x)}{[-\ln(x)]}\no\\
&=
\begin{cases}
\displaystyle{\lim_{x\downarrow0}\dfrac{- \ln(x) - \psi(-z) - 2 \gamma_E + \Oh(x|\ln(x)|)}{[-\ln(x)]}},& z\in \bbC \backslash \{0\},  \\
\displaystyle{\lim_{x\downarrow0}\frac{- \ln(x) \big\{1 + C_0 [- \ln(x)]^{-1} + x + \Oh\big(x^2\big)\big\}}{[-\ln(x)]}},& z=0,
\end{cases}\no\\
&=1,\quad z\in \bbC,\lb{23.6.99y}\\
\wti y_{1,1}'(z,0) &= \lim_{x\downarrow0}\big\{ y_{1,1}(z,x) - \wti y_{1,1}(z,0)[-\ln(x)]\big\}\no\\
&= \lim_{x\downarrow0}y_{1,1}(z,x)=1,\quad z\in \bbC,\lb{23.6.100y}\\
\wti y_{2,1}'(z,0)&=\lim_{x\downarrow0}\big\{y_{2,1}(z,x) - \wti y_{2,1}(z,x)[-\ln(x)]\big\}\no\\
&=\lim_{x\downarrow0}\big\{y_{2,1}(z,x)+\ln(x)\big\}\no\\
&=
\begin{cases}
\displaystyle{\lim_{x\downarrow0}\big\{- \ln(x) - \psi(-z) - 2 \gamma_E + \Oh(x|\ln(x)|)+\ln(x)\big\}},& z\in \bbC \backslash \{0\},\\
\displaystyle{\lim_{x\downarrow0}\big[- \ln(x) \big\{1 + C_0 [- \ln(x)]^{-1} + x + \Oh\big(x^2\big)\big\}+\ln(x)\big]},&z=0,
\end{cases}\no\\
&=
\begin{cases}
-\psi(-z) - 2 \gamma_E, &z\in \bbC \backslash \{0\},\\
C_0,&z=0.
\end{cases}     \lb{23.6.101y}
\end{align}
(Again we used the notation, $\psi (\dott) = \Gamma'(\dott)/\Gamma(\dott)$.) As a result,
\begin{equation}\lb{23.6.102y}
\wti y_{1,1}(z,0)\wti y_{2,1}'(z,0) - \wti y_{1,1}'(z,0)\wti y_{2,1}(z,0) 
= 0\cdot\wti y_{2,1}'(z,0)-1\cdot 1 = -1, 
\quad z\in \bbC. 
\end{equation}

The identities \eqref{23.6.64}, \eqref{23.6.65} and \eqref{23.6.66} in conjunction with \eqref{23.6.98y}--\eqref{23.6.101y} and \eqref{23.6.102y} imply
\begin{align}
c_{\phi,1,1}(z) &= \wti y_{2,1}(z,0) = 1,\no\\
c_{\phi,2,1}(z) &= -\wti y_{1,1}(z,0) = 0,\no\\
c_{\theta,1,1}(z) &= -\wti y_{2,1}'(z,0) =
\begin{cases}
\psi(-z) + 2 \gamma_E, & z\in \bbC \backslash \{0\},\\
-C_0,&z=0,
\end{cases}\no\\
c_{\theta,2,1}(z) &= \wti y_{1,1}'(z,0)=1,\quad z\in \bbC,\lb{23.6.103y}
\end{align}
and
\begin{align}
\phi_{0,1}(z,x) &= y_{1,1}(z,x),\no\\
\theta_{0,1}(z,x) &=
\begin{cases}
[\psi(-z) + 2 \gamma_E] y_{1,1}(z,x) + y_{2,1}(z,x),& z\in \bbC \backslash \{0\},    \lb{23.6.104y}\\
-C_0y_{1,1}(z,x) + y_{2,1}(z,x),& z= 0,
\end{cases}\\
&\hspace*{4.2cm}z\in \bbC, \; x\in (0,\infty).\no
\end{align}

The requirement in \eqref{23.6.68} may be recast as
\begin{align}
\begin{split} 
\{[\psi(-z) + 2 \gamma_E + m_{0,1,Lag}(z)]y_{1,1}(z,\,\cdot\,) + y_{2,1}(z,\,\cdot\,)\}\in L^2\big((0,\infty);e^{-x}\, dx\big),&\lb{23.6.105y}\\
z\in \rho(T_{F,1,Lag})\backslash\{0\},&
\end{split} 
\end{align}
or as, 
\begin{align}
\begin{split} 
\int^{\infty}|[\psi(-z) + 2 \gamma_E + m_{0,1,Lag}(z)]y_{1,1}(z,x) + y_{2,1}(z,x)|^2e^{-x}\, dx<\infty,&     \lb{23.6.106y}\\
z\in \rho(T_{F,1,Lag})\backslash\{0\}.&
\end{split} 
\end{align}
The asymptotic relations \eqref{23.6.52A} and 
\begin{equation}\lb{23.6.103}
U(-z,b;x) \underset{x \uparrow \infty}{=} x^{-z}\big[1 + \Oh\big(|x|^{-1}\big)\big],
\quad z\in \bbC, \; b\in \bbC,
\end{equation}
imply that the integrand in \eqref{23.6.106y} behaves at $\infty$ like
\begin{align}
&|[\psi(-z) + 2 \gamma_E + m_{0,1,Lag}(z)]y_{1,1}(z,x) + y_{2,1}(z,x)|^2e^{-x}\no\\
&\quad=|[\psi(-z) + 2 \gamma_E + m_{0,1,Lag}(z)]F(-z,1;x) + \Gamma(-z)U(-z,1;x)|^2e^{-x}\no\\
&\quad=\bigg|[\psi(-z) + 2 \gamma_E + m_{0,1,Lag}(z)]\frac{1}{\Gamma(-z)}e^xx^{-1-z} + \Gamma(-z)x^{-z}\bigg|^2e^{-x}\no\\
&\quad=\bigg|[\psi(-z) + 2 \gamma_E + m_{0,1,Lag}(z)]\frac{1}{\Gamma(-z)} + \Gamma(-z)xe^{-x}\bigg|^2x^{-2(1+\Re(z))}e^{x}.     \lb{23.6.104}
\end{align}
The expression \eqref{23.6.104} is integrable near $\infty$ with respect to Lebesgue measure $dx$ if and only if 
\begin{align}
[\psi(-z) + 2 \gamma_E + m_{0,1,Lag}(z)]/\Gamma(-z)=0,
\end{align}
so that 
\begin{equation}
m_{0,1,Lag}(z) = - \psi(-z) - 2 \gamma_E,   \quad z\in \rho(T_{F,1,Lag}),    \lb{23.6.110y}
\end{equation} 
in agreement with \cite[eq.~(8)]{De98} (replacing $\Gamma(\dott)$ by $\ln(\Gamma(\dott))$ correcting an obvious  typographical error) and the choice of $\Gamma_0, \Gamma_1$ therein. 

The representation (cf.\ \cite[p.~13]{MOS66})
\begin{equation}
- \psi(-z) = \gamma_E + \sum_{\bbN_0} \bigg(\f{1}{n -z} - \f{1}{n+1}\bigg)
\end{equation}
proves the Nevanlinna--Herglotz property of $m_{0,1,Lag}(\dott)$, again in agreement with \eqref{23.5.12}. 

By \eqref{23.6.110y}, the poles of $m_{0,1,Lag}(\,\cdot\,)$ are $\bbN_0$.  In particular,
\begin{equation}
\sigma(T_{F,\beta,Lag}) = \begin{cases} \{n+1-\beta \}_{n\in \bbN_0}, & \beta \in (0,1), \\ 
\bbN_0, & \beta \in [1,2). 
\end{cases}
\end{equation}

For $\beta \in (0,1),$ the corresponding eigenfunction
for the eigenvalue $\lambda _{n}=n+1-\beta $ is 
\begin{align}
\begin{split} 
y_{n,\beta }(x) &=x^{1-\beta }\text{ }_{1}F_{1}(n-2\beta +2;2-\beta ;x) \\
&=\sum_{k=0}^{\infty }\frac{(n-2\beta +2)_{k}}{(2-\beta )_{k}k!}x^{k},  \quad
n\in \mathbb{N}_{0}, \; x \in (0,\infty), 
\end{split} 
\end{align} 
with $(\gamma)_0 = 1$, $(\gamma)_k = \Gamma(\gamma+k)/\Gamma(\gamma)$, $k \in \bbN$, denoting Pochhammer's symbol.

For $\beta \in \lbrack 1,2),$ the corresponding eigenfunction for $\lambda
_{n}=n$ is 
\begin{equation} 
y_{n,\beta }(x)=L_{n}^{\beta -1}(x), \quad n\in \mathbb{N}_{0}, \; x \in (0,\infty), 
\end{equation} 
the $n^{th}$ Laguerre polynomial. 

\begin{remark} \lb{23.r6.2}
As an interesting point of comparison and an illustration of the significant effect a change of boundary conditions may have on spectra, we recall that the self-adjoint realization $T_{N,\beta,Lag}$ of $\tau_{\beta}$ obtained by restricting $T_{max,\beta,Lag}$ to the ``Neumann'' (rather than the Dirichlet, resp., Friedrichs) domain
\begin{equation}
\dom(T_{N,\beta,Lag}) = \big\{y\in \dom(T_{max,\beta,Lag})\,\big|\, \wti y'(0)=0\big\},\quad \beta\in (0,2),
\end{equation}
has spectrum equal to $\bbN_0$ for all $\beta\in (0,2)$ (cf., e.g., \cite[Sect.~27]{Ev05}).  One notes that the boundary condition ``$[y,1](0)=0$'' employed in \cite[Sect.~27]{Ev05} is equivalent to $\wti y'(0)=0$.  In fact, $T_{N,\beta,Lag}$ has the classical Laguerre polynomials as eigenfunctions.  Thus, unlike $\sigma(T_{F,\beta,Lag})$, the spectrum of $T_{N,\beta,Lag}$ is independent of $\beta\in (0,2)$. \hfill $\diamond$
\end{remark}

\begin{remark} \lb{23.r6.3} 
In this paper, we have considered the Laguerre
parameter $\beta $ belonging to the range $(0,2)$. We note that, for $\beta \geq 2$, the 
Laguerre expression $\tau _{\beta ,Lag}$ is limit point at both 
$x=0$ and $x=\infty $ in $L^{2}((0,\infty );x^{\beta -1}e^{-x})$. 
Consequently, in this case, there is a unique self-adjoint extension of the
minimal operator $T_{\min ,\beta ,Lag}$ and the spectrum of this operator is 
$\mathbb{N}_{0}$ just like the case $\beta \in [1,2)$. Moreover, for
the eigenvalue $\lambda _{n}=n$, $n \in \bbN_0$, the corresponding eigenfunction is 
\begin{equation} 
y_{n,\beta }(x)=L_{n}^{\beta -1}(x), \quad n \in \bbN_0, \; x \in (0,\infty), 
\end{equation} 
just as in the case of $\beta \in [1,2)$.  \hfill $\diamond$
\end{remark}

\begin{remark} \lb{r23.6.4} Finally, we briefly discuss the non-classical Laguerre
case for $\beta \leq 0$. First, when $-\beta \in \mathbb{N}_{0}$, the
Laguerre expression $\tau _{\beta ,Lag}$ is limit point at both endpoints $x=0$ and $x=\infty $ 
in $L^2\big((0,\infty );x^{\beta -1}e^{-x}\big)$. In this
case, the unique self-adjoint extension of $T_{\min ,\beta ,Lag}$ has
spectrum $\{n\in \mathbb{N}_{0}\}_{n\geq -\beta }$ and the corresponding
eigenfunctions, $\big\{L_{n}^{\beta -1}(\cdot )\big\}_{n\geq -\beta }$, form a
complete orthogonal set in $L^2\big((0,\infty );x^{\beta -1}e^{-x}\big);$ see \cite{ELW04}
for details. The case $\beta \leq 0$ but $-\beta \notin \mathbb{N}_{0}$ is
studied in remarkable detail by Derkach in \cite{De98}; see also \cite{Kr79}. In this
situation, the appropriate function space setting is a Pontryagin space $\prod (\beta )$, 
not a Hilbert function space. Derkach shows that, if $-n<\beta <-n+1,$ the Laguerre polynomials 
$\big\{L_{m}^{\beta -1}\big\}_{m\geq 0}$
are orthogonal with respect to the indefinite inner product 
\begin{equation}
(f,g)=\int_{0}^{\infty} dx \, x^{\beta -1}\bigg( e^{-x}\overline{f(x)}g(x)  
-\sum_{j=0}^{n-1}(e^{-x} \overline{f}g)^{(j)}(0)\frac{x^{j}}{j!}\bigg).
\end{equation}
In addition, Derkach \cite{De98} determined the corresponding Weyl--Titchmarsh--Kodaira $m$-function 
as well as the self-adjoint operator $T_{\beta }$ in  $\prod
(\beta )$ having the Laguerre polynomials as eigenfunctions. \hfill $\diamond$
\end{remark}

\medskip

\noindent 
{\bf Acknowledgments.} We gratefully acknowledge very detailed comments by Mark Ashbaugh, Christian Berg, Charles Fulton, Hubert Kalf, and Jonathan Stanfill on an earlier version of this paper. In particular, 
Mark Ashbaugh, Charles Fulton, and Jonathan Stanfill directed us to various shortcomings in our treatment of the Legendre and Laguerre examples in Subsections \ref{23.e7.3.13} and \ref{23.e7.3.14}, and Christian Berg kindly pointed out to us the insufficiency of an earlier attempt to demonstrate the Nevanlinna--Herglotz property of \eqref{23.6.81} and \eqref{23.6.91}.  We also thank Mark Ashbaugh for very extensive correspondence regarding Section \ref{23.s6}, and Charles Fulton for numerous comments on the bulk of this paper. Finally, we also thank Dale Frymark, Annemarie Luger, and Tony Zettl for helpful discussions and correspondence. 

\medskip
 
 
\end{document}